\documentclass[a4paper,11pt, reqno, oneside]{amsart}

\usepackage[utf8]{inputenc}
\usepackage{amsmath, amsfonts, amssymb, amsthm, mathtools, bbm, enumerate}

\usepackage[all]{xy}
\usepackage[height=22.5cm, width=15.5cm]{geometry}
\usepackage[active]{srcltx}

\theoremstyle{plain}
\newtheorem{thm}{Theorem}[section]
\newtheorem*{thm*}{Theorem}
\newtheorem{lemma}[thm]{Lemma}
\newtheorem{prop}[thm]{Proposition}
\newtheorem{cor}[thm]{Corollary}
\theoremstyle{definition}
\newtheorem{defi}[thm]{Definition}
\newtheorem*{defi*}{Definition}
\newtheorem{rmk}[thm]{Remark}
\newtheorem{ex}[thm]{Example}

\DeclareMathOperator{\Aut}{Aut}

\providecommand{\bysame}{\leavevmode\hbox to3em{\hrulefill}\thinspace}
\providecommand{\MR}{\relax\ifhmode\unskip\space\fi MR }

\providecommand{\href}[2]{#2}
\begin{document}

\title{Automorphism groups of compact complex supermanifolds}
\author{Hannah Bergner}
\author{Matthias Kalus}

\address{Fakult\"at f\"ur Mathematik, Ruhr-Universit\"at Bochum,
Universit\"atsstr. 150, D-44780 Bochum, Germany}
\email{Hannah.Bergner-C9q@rub.de}
\email{Matthias.Kalus@rub.de}

\thanks{Financial support by SFB/TR 12 ``Symmetries and Universality in 
Mesoscopic Systems'' of the DFG is gratefully acknowledged.}

\begin{abstract}
Let $\mathcal M$ be a compact complex supermanifold. We prove that the 
set $\Aut_{\bar 0}(\mathcal M)$ of automorphisms of $\mathcal M$
can be endowed with the structure of a complex Lie group acting holomorphically 
on $\mathcal M$, so that its Lie
algebra is isomorphic to the Lie algebra of even holomorphic super vector fields
on $\mathcal M$. 
Moreover, we prove the existence of a complex Lie supergroup $\Aut(\mathcal M)$
acting holomorphically on $\mathcal M$ and satisfying a universal property.
Its underlying Lie group is $\Aut_{\bar 0}(\mathcal M)$ and its Lie superalgebra
is the Lie superalgebra of holomorphic super vector fields on $\mathcal M$.
This generalizes the classical theorem by Bochner and Montgomery that
the automorphism group of a compact complex manifold is a complex Lie group.
Some examples of automorphism groups of complex supermanifolds over 
$\mathbb P_1(\mathbb C)$ are provided.
\end{abstract}

\maketitle
\begin{center}
  {\bf Keywords:} compact complex supermanifold; automorphism group\\
  \vskip0.2cm
   {\bf MSC2010:}   54H15, 32M05, 32C11 \\
\end{center}

\section{Introduction}
The automorphism group of a compact complex manifold $M$ carries the structure 
of a complex Lie group which acts holomorphically on $M$ and whose Lie algebra 
consists of the holomorphic vector fields on $M$ 
(see \cite{BMGroupsAnalyticManifold}). 
In this article, we investigate how this result can be extended to the
category of compact complex supermanifolds.

Let $\mathcal M$ be a compact complex supermanifold, i.e. a complex 
supermanifold whose underlying manifold is compact. 
An automorphism of $\mathcal M$ is a biholomorphic morphism 
$\mathcal M\rightarrow \mathcal M$.
A first candidate for the automorphism group of such a supermanifold is
be the set of automorphisms, which we denote by $\Aut_{\bar 0}(\mathcal M)$.
However, every automorphism $\varphi$ of a supermanifold $\mathcal M$ (with 
structure sheaf $\mathcal O_\mathcal M$) is ``even'' in the sense that its 
pullback $\varphi^*:\mathcal O_\mathcal M\rightarrow\tilde\varphi_*
(\mathcal O_\mathcal M)$ is a parity-preserving morphism.
Therefore, we can (at most) expect this set of automorphisms of $\mathcal M$ to
carry the structure of a classical Lie group if we require its action on 
$\mathcal M$ to be smooth or holomorphic. This way we will not receive a Lie 
supergroup of positive odd dimension.
We will prove that the topological group 
$\Aut_{\bar 0}(\mathcal M)$, endowed with an analogue of the compact-open 
topology, carries the structure of a complex Lie group such that the action on 
$\mathcal M$ is holomorphic and its Lie algebra is the Lie algebra of even 
holomorphic super vector fields on $\mathcal M$.
It should be noted that the group $\Aut_{\bar 0}(\mathcal M)$ is in general
different from the group $\Aut(M)$ of automorphisms of the underlying 
manifold~$M$. There is a group homomorphism $\Aut_{\bar 0}(\mathcal M)
\rightarrow \Aut(M)$ given by assigning the underlying map to an automorphism of
the supermanifold; this group homomorphism is in general neither injective nor
surjective.

We will find the automorphism group of a compact complex supermanifold~
$\mathcal M$
to be a complex Lie supergroup which acts holomorphically on $\mathcal M$ and 
satisfies a universal property. In analogy to the classical case, its Lie 
superalgebra is the Lie superalgebra of holomorphic super vector fields 
on $\mathcal M$, and the underlying Lie group is 
$\Aut_{\bar 0}(\mathcal M)$, the group of automorphisms of $\mathcal M$.
Using the equivalence of complex Harish-Chandra pairs and complex Lie 
supergroups (see \cite{Vishnyakova}), we construct the 
appropriate automorphism Lie supergroup of $\mathcal M$.

More precisely, the outline of this article is the following:
First, we introduce a topology on the set $\Aut_{\bar 0}(\mathcal M)$ of 
automorphisms on a compact complex supermanifold $\mathcal M$
(c.f. Section~$\S 3$). This topology is analogue of the compact-open topology
in the classical case, which coincides in the case of a compact complex manifold
with the topology of uniform convergence. We prove that the topological space
$\Aut_{\bar 0}(\mathcal M)$ with composition and inversion of automorphisms
as group operations is a locally compact topological group which satisfies 
the second axiom of countability.

In Section~$\S 4$, the non-existence of small subgroups of 
$\Aut_{\bar 0}(\mathcal M)$ is proven, which means that there exists a
neighbourhood of the identity in $\Aut_{\bar 0}(\mathcal M)$ with the property
that this neighbourhood does not contain any non-trivial subgroup.
A result on the existence of Lie group structures on locally compact 
topological groups without small subgroups (see \cite{Yamabe}) then implies that 
$\Aut_{\bar 0}(\mathcal M)$ carries the structure of real Lie group.

Then, continuous one-parameter subgroups of $\Aut_{\bar 0}(\mathcal M)$
and their action on the supermanifold~$\mathcal M$ are studied (see Section~
$\S 5$). This is done in order to obtain results on the regularity of
the $\Aut_{\bar 0}(\mathcal M)$-action on $\mathcal M$ and characterize
the Lie algebra of $\Aut_{\bar 0}(\mathcal M)$.
We prove that the action of each continuous one-parameter subgroup of 
$\Aut_{\bar 0}(\mathcal M)$ on $\mathcal M$ is analytic.
As a corollary we get that the Lie algebra of $\Aut_{\bar 0}(\mathcal M)$ is
isomorphic to the Lie algebra $\mathrm{Vec}_{\bar 0}(\mathcal M)$ of 
even holomorphic super vector fields on $\mathcal M$, and
$\Aut_{\bar 0}(\mathcal M)$ carries the structure of a complex Lie group
so that its natural action on $\mathcal M$ is holomorphic.

Next, we show that the Lie superalgebra $\mathrm{Vec}(\mathcal M)$ of 
holomorphic super vector fields on a compact complex supermanifold $\mathcal M$
is finite-dimensional (see Section~$\S 6$). Since $\Aut_{\bar 0}(\mathcal M)$
carries the structure of a complex Lie group, we already know that 
$\mathrm{Vec}_{\bar 0}(\mathcal M)$, the even part of 
$\mathrm{Vec}(\mathcal M)$, is finite-dimensional.
The key point in the proof in the case of a split supermanifold~$\mathcal M$ 
is that the tangent sheaf of $\mathcal M$ is a coherent sheaf of 
$\mathcal O_M$-modules on the compact complex manifold $M$ with
$\mathcal O_M$ as the sheaf of holomorphic functions on $M$.

Let $\alpha$ denote the action of $\Aut_{\bar 0}(\mathcal M)$ on the Lie
superalgebra $\mathrm{Vec}(\mathcal M)$ by conjugation:
$\alpha(\varphi)(X)=\varphi_*(X)=(\varphi^{-1})^*\circ X\circ\varphi^*$ for 
$\varphi\in\Aut_{\bar 0}(\mathcal M)$, $X\in \mathrm{Vec}(\mathcal M)$.
The restriction of this representation~$\alpha$ to $\mathrm{Vec}_{\bar 0}
(\mathcal M)$, the even part of the Lie superalgebra $\mathrm{Vec}(\mathcal M)$,
coincides with the adjoint action of the Lie group $\Aut_{\bar 0}(\mathcal M)$
on its Lie algebra, which is isomorphic to $\mathrm{Vec}_{\bar 0}(\mathcal M)$.
Hence $\alpha$ defines a Harish-Chandra pair
$(\Aut_{\bar 0}(\mathcal M), \mathrm{Vec}(\mathcal M))$. 
The equivalence between Harish-Chandra pairs and complex Lie supergroups
allows us to define the automorphism Lie supergroup of a compact complex
supermanifold as follows (see Definition~\ref{defi: automorphism group}):

\begin{defi*}[Automorphism Lie supergroup]
 Define the automorphism group $\Aut(\mathcal M)$ of a compact
 complex supermanifold to be the unique complex Lie supergroup associated
 to the Harish-Chandra pair 
 $(\Aut_{\bar 0}(\mathcal M), \mathrm{Vec}(\mathcal M))$ with representation
 $\alpha$.
\end{defi*}

The natural action of the automorphism Lie supergroup $\Aut(\mathcal M)$
on $\mathcal M$ is holomorphic, i.e. we have a morphism 
$\Psi:\Aut(\mathcal M)\times\mathcal M\rightarrow\mathcal M$ of complex 
supermanifolds.
The automorphism Lie supergroup $\Aut(\mathcal M)$ satisfies the following
universal property (see Theorem~\ref{thm: universal property}):

\begin{thm*}
 If $\mathcal G$ is a complex Lie supergroup with a holomorphic action
 $\Psi_{\mathcal G}:\mathcal G\times \mathcal M\rightarrow \mathcal M$ 
 on $\mathcal M$, then there is a unique morphism $\sigma:\mathcal G\rightarrow 
 \Aut(\mathcal M)$ of Lie supergroups such that the diagram
$$\xymatrix{\mathcal G\times\mathcal M
\ar[dr]_-{\sigma\times \mathrm{id}_\mathcal M }
 \ar[rr]^{\Psi_\mathcal G}&&\mathcal M\\
 &\Aut(\mathcal M)\times \mathcal M \ar[ru]_-{\Psi}&&
}$$
is commutative.
\end{thm*}

The automorphism Lie supergroup of a compact complex supermanifold is the
unique complex Lie supergroup satisfying the preceding universal property.

In the classical case, another class of complex manifolds where the 
automorphism group carries the structure of a Lie group is given by the
bounded domains in $\mathbb C^m$ (see \cite{CartanOeuvres}).
An analogue statement is false in the case of supermanifolds.
In Section~$\S 8$, we give an example showing that in the case of a complex
supermanifold~$\mathcal M$ whose underlying manifold is a bounded domain in 
$\mathbb C^m$ there does in general not exist a Lie supergroup acting on 
$\mathcal M$ and satisfying the universal property of the preceding theorem.

In Section $\S 9$, the automorphism group $\Aut(\mathcal M)$ or
its underlying Lie group $\Aut_{\bar 0}(\mathcal M)$ are calculated 
for some supermanifolds $\mathcal M$ with underlying manifold 
$M=\mathbb P_1\mathbb C$.

\section{Preliminaries and Notation}
Throughout, we work with the ``Berezin-Le{\u\i}tes-Kostant-approach'' to
supermanifolds (c.f. \cite{Berezin}, \cite{Leites}, and \cite{Kostant}).
If a supermanifold is denoted by a calligraphic letter $\mathcal M$,
then we denote the underlying manifold by the corresponding uppercase standard
letter $M$, and the structure sheaf by $\mathcal O_\mathcal M$.
We call a supermanifold~$\mathcal M$ compact if its underlying manifold 
$M$ is compact.
By a complex supermanifold we mean a supermanifold $\mathcal M$ with 
structure sheaf~$\mathcal O_\mathcal M$ which is locally, on small enough
open subsets $U\subset M$, isomorphic to
$\mathcal O_U\otimes \bigwedge \mathbb C^n$, where
$\mathcal O_U$ denotes the sheaf of holomorphic functions on $U$.
For a morphism $\varphi:\mathcal M\rightarrow\mathcal N$ between supermanifolds
$\mathcal M$ and $\mathcal N$, the underlying map $M\rightarrow N$ is denoted by 
$\tilde\varphi$ and its pullback by $\varphi^*:\mathcal O_\mathcal N
\rightarrow \tilde\varphi_*\mathcal O_\mathcal M$. 
An automorphism of a complex supermanifold $\mathcal M$ is a biholomorphic
morphism $\mathcal M\rightarrow \mathcal M$.

For a complex supermanifold $\mathcal M$, let $\mathcal T_\mathcal M$ denote 
the tangent sheaf of $\mathcal M$. The Lie superalgebra of holomorphic 
vector fields on $\mathcal M$ is 
$\mathrm{Vec}(\mathcal M)=\mathcal T_\mathcal M(M)$,
it consists of the subspace $\mathrm{Vec}_{\bar 0}(\mathcal M)$ of
even and the subspace $\mathrm{Vec}_{\bar 1}(\mathcal M)$ of 
odd super vector fields on $\mathcal M$.

 Let $\mathcal M$ be a complex supermanifold of dimension $(m|n)$, and let 
 $\mathcal I_\mathcal M$ be the subsheaf of ideals generated by the odd elements 
 in the structure sheaf $\mathcal O_\mathcal M$ of a supermanifold $\mathcal M$.
 As described in \cite{Onishchik}, we have the filtration
 $$\mathcal O_\mathcal M=(\mathcal I_\mathcal M)^0\supset 
 (\mathcal I_\mathcal M)^1\supset(\mathcal I_\mathcal M)^2\supset \ldots
 \supset (\mathcal I_\mathcal M)^{n+1}=0$$
 of the structure sheaf $\mathcal O_\mathcal M$ by the powers of 
 $\mathcal I_\mathcal M$.
 Define the quotient sheaves 
 $\text{gr}_k(\mathcal O_\mathcal M)=(\mathcal I_\mathcal M)^k/
 (\mathcal I_\mathcal M)^{k+1}$.
 This gives rise to the $\mathbb Z$-graded sheaf 
 $\text{gr}\,\mathcal O_\mathcal M
 ={\textstyle \bigoplus_k} \text{gr}_k(\mathcal O_\mathcal M)$. 
 Further $\text{gr}\,\mathcal M= (M,\text{gr}\,\mathcal O_\mathcal M)$ is a
 split complex supermanifold of the same dimension as $\mathcal M$. 
 
 \smallskip\noindent
 Note that $E:=\text{gr}_1(\mathcal O_\mathcal M)$ defines a vector bundle on 
 $M$. An automorphism $\varphi$ of $\mathcal M$ yields a pullback $\varphi^\ast$ 
 on $\mathcal O_\mathcal M$. Following \cite{Green}, its reduction to the
 $\mathcal O_M$-module $E$ yields a morphism of vector bundles 
 $\varphi_0\in Aut(E)$ over the reduction $\tilde\varphi\in Aut(M)$. 
 Note that by \cite{Morimito}, $Aut(E)$ is a complex Lie group. On local 
 coordinate domains $U,V$ with $\varphi(U)\subset V$ we can identify 
 $\mathcal O_\mathcal M|_V\cong \Gamma_{\Lambda E}|_V$ and 
 $\mathcal O_\mathcal M|_U\cong \Gamma_{\Lambda E}|_U$ and following 
 \cite{Rothstein} decompose  $\varphi^\ast=\varphi_0^\ast\exp(Y)$ with 
 $\mathbb Z$-degree preserving automorphism 
 $\varphi_0^\ast:\Gamma_{\Lambda E}|_V \to \Gamma_{\Lambda E}|_U$ 
 induced by $\varphi_0$. Here $Y$ is an even superderivation on 
 $\Gamma_{\Lambda E}|_V$ increasing the $\mathbb Z$-degree by $2$ or more. 
 Note that the exponential series is finite since $Y$ is nilpotent.

\section{The topology on the group of automorphisms}
Let $\mathcal M$ be a compact complex supermanifold.
An automorphism of $\mathcal M$ is a biholomorphic morphism 
$\varphi:\mathcal M\rightarrow\mathcal M$.
Denote by $\Aut_{\bar 0}(\mathcal M)$ the set of automorphisms of $\mathcal M$.

In this section, a topology on $\Aut_{\bar 0}(\mathcal M)$ is introduced, 
which  generalizes the compact-open topology and topology of compact 
convergence of the classical case. Then we show that 
$\Aut_{\bar 0}(\mathcal M)$ is a locally compact topological group with respect
to this topology. 

Let $K\subseteq M$ be a compact subset such that there are local odd coordinates 
$\theta_1,\ldots, \theta_n$ for $\mathcal M$
on an open neighbourhood of $K$.
Moreover, let $U\subseteq M$ be open and $f\in \mathcal O_\mathcal M(U)$,
and let $U_{\nu}$ be open subsets of $\mathbb C$ for $\nu\in (\mathbb Z_2)^n$.
Let $\varphi:\mathcal M\rightarrow\mathcal M$ be an automorphism with
$\tilde\varphi(K)\subseteq U$.
Then there are holomorphic functions $\varphi_{f,\nu}$ on
a neighbourhood of $K$ such that
$$\varphi^*(f)=\sum_{\nu\in (\mathbb Z_2)^n}\varphi_{f,\nu} \theta^\nu.$$
Let
$$\Delta(K, U, f,\theta_j, U_{\nu})
=\{\varphi\in \Aut_{\bar 0}(\mathcal M)|\,\tilde{\varphi}(K)\subseteq U,
\, \varphi_{f,\nu}(K)\subseteq U_{\nu}\},$$
and endow $\Aut_{\bar 0}(\mathcal M)$ with the topology generated by sets of 
this form, i.e. the sets of the form $\Delta(K, U,f,\theta_j, U_{\nu})$ form a 
subbase of the topology.

\begin{rmk}
In particular, the subsets of the form
$\Delta(K,U)=\{\varphi\in \Aut_{\bar 0}(\mathcal M)
|\,\tilde{\varphi}(K)\subseteq U\}$
are open for $K\subseteq M$ compact and $U\subseteq M$ open.
Hence the map $\Aut_{\bar 0}(\mathcal M)\rightarrow \Aut(M)$,
associating to an automorphism $\varphi$ of $\mathcal M$ the underlying
automorphism~$\tilde{\varphi}$ of $M$, is continuous.
\end{rmk}

\begin{rmk}
 The group $\Aut_{\bar 0}(\mathcal M)$ endowed with the above topology is a 
 second-countable Hausdorff space since $M$ is second-countable.
\end{rmk}

 Let $U\subseteq M$ be open. Then we can define a topology on 
 $\mathcal O_\mathcal M(U)$ as follows:
 If $K\subseteq U$ is compact such that there exist odd coordinates 
 $\theta_1,\ldots, \theta_n$ on a neighbourhood of $K$, write
 $f\in \mathcal O_\mathcal M(U)$ on $K$ as $f=\sum_{\nu}f_\nu \theta^\nu$.
 Let $U_\nu\subseteq \mathbb C$ be open subsets.
 Then define a topology on $\mathcal O_\mathcal M(U)$ by requiring that the sets
 of the form $\{f\in \mathcal O_\mathcal M(U)|\,f_\nu(K)\subseteq U_\nu\}$ are
 a subbase of the topology.
 A sequence of functions $f_k$ converges to $f$ if and only if in all local 
 coordinate domains with odd coordinates $\theta_1,\ldots,\theta_n$ and  
 $f_k=\sum_{\nu}f_{k,\nu}\theta^\nu$, $f=\sum_{\nu}f_\nu\theta^\nu$, the 
 coefficient functions $f_{k,\nu}$ converge uniformly to $f_\nu$ on compact
 subsets.
 Note that for any open subsets $U_1, U_2\subseteq M$ with $U_1\subset U_2$ the
 restriction map 
 $\mathcal O_\mathcal M(U_2)\rightarrow \mathcal O_\mathcal M(U_1)$, 
 $f\mapsto f|_{U_1}$, is continuous.

Using Taylor expansion (in local coordinates) of automorphisms of $\mathcal M$
we can deduce the following lemma:

\begin{lemma}
 A sequence of automorphisms $\varphi_k:\mathcal M\rightarrow \mathcal M$
 converges to an automorphism $\varphi:\mathcal M\rightarrow \mathcal M$ with 
 respect to the topology of $\Aut_{\bar 0}(\mathcal M)$ if and only if the 
 following condition is satisfied:
 For all $U,V\subseteq M$  open subsets of $M$ such that $V$ contains the
 closure of $\tilde{\varphi}(U)$, there is an $N\in \mathbb N$ such that 
 $\tilde{\varphi_k}(U)\subseteq V$ for all $k\geq N$. 
 Furthermore, for any $f\in \mathcal O_\mathcal M(V)$ the sequence
 $(\varphi_k)^*(f)$ converges to $\varphi^*(f)$ on $U$ in the topology of
 $\mathcal O_\mathcal M(U)$.
\end{lemma}

\begin{lemma}\label{le123}
 If $U, V\subseteq M$ are open subsets, $K\subseteq M$ is compact with
 $V\subseteq K$, then the map 
 $$\Delta(K,U)\times\mathcal O_\mathcal M(U)
 \rightarrow \mathcal O_\mathcal M(V),\, (\varphi,f)\mapsto \varphi^*(f)$$
 is continuous.
\end{lemma}

\begin{proof}
 Let $\varphi_k\in \Delta(K,U)$ be a sequence of automorphisms of $\mathcal M$
 converging to $\varphi\in \Delta(K,U)$, and $f_l\in \mathcal O_\mathcal M(U)$
 a sequence converging to $f\in \mathcal O_\mathcal M(U)$. 
 Choosing appropriate local coordinates and using Taylor expansion of
 the pullbacks $(\varphi_k)^*(f_l)$, it can be shown that
 $(\varphi_k)^*(f_l)$ converges to $\varphi^*(f)$ as $k,l\to\infty$.
 This uses that the derivatives of a sequence of uniformly converging 
 holomorphic functions also uniformly converge.
\end{proof}

\begin{lemma}\label{lemma: locally compact}
 The topological space $\Aut_{\bar 0}(\mathcal M)$ is locally compact.
\end{lemma}

\begin{proof}
 Let $\psi\in \Aut_{\bar 0}(\mathcal M)$.
 For each fixed $x\in M$ there are open neighbourhoods $V_x$ and $U_x$ of $x$ 
 and $\tilde{\psi}(x)$ respectively such that $\tilde{\psi}(K_x)\subseteq U_x$ 
 for $K_x:=\overline{V}_x$.
 We may additionally assume  that there are local odd coordinates 
 $\xi_1,\ldots, \xi_n$ for $\mathcal M$ on $U_x$, and $\theta_1,\ldots, \theta_n$
 local odd coordinates on an open neighbourhood of $K_x$.
 For any automorphism $\varphi:\mathcal M\rightarrow \mathcal M$ with
 $\tilde{\varphi}(K_x)\subseteq U_x$, let $\varphi_{j,k}$, $\varphi_{j,\nu}$ 
 (for $||\nu||=||(\nu_1,\ldots,\nu_n)||=\nu_1+\ldots+\nu_n\geq 3$) be local
 holomorphic functions such that 
 $$\varphi^*(\xi_j)=\sum_{k=1}^n \varphi_{j,k} \theta_k
 +\sum_{||\nu||\geq 3} \varphi_{j,\nu}\theta^\nu.$$
 Choose bounded open subsets $U_{j,k}, U_{j,\nu}\subset \mathbb C$, 
 such that $\psi_{j,k}(x)\in U_{j,k}$ and $\psi_{j,\nu}(x)\in U_{j,\nu}$. 
 Since $\psi$ is an automorphism, we have 
 $$\det\left((\psi_{j,k}(y))_{1\leq j,k\leq n}\right)\neq 0$$ 
 for all $y\in K_x$. For later considerations shrink $U_{j,k}$ such that 
 $\det(C)\neq 0$ for all $C=(c_{j,k})_{1\leq j,k\leq n}$ with 
 $c_{j,k}\in U_{j,k}$.
 After shrinking $V_x$ we may assume $\psi_{j,k}(K_x)\subseteq U_{j,k}$ and 
 $\psi_{j,\nu}(K_x)\subseteq U_{j,\nu}$.
 Hence $\psi$ is contained in the set 
 $\Theta(x)=\{\varphi\in \Aut_{\bar 0}(\mathcal M)\,|\,\tilde{\varphi}(K_x)
 \subseteq \overline{U}_x,\,\varphi_{j,k}(K_x)\subseteq \overline{U}_{j,k},
 \varphi_{j,\nu}(K_x) \subseteq \overline{U}_{j,\nu}\}$,
 which contains an open neighbourhood of $\psi$. 
 Since $M$ is compact, $M$ is covered by finitely many of the sets $V_x$, 
 say $V_{x_1},\ldots,V_{x_l}$. 
 Then $\psi$ is contained in $\Theta=\Theta(x_1)\cap\ldots\cap \Theta(x_l)$.
 We will now prove that $\Theta$ is sequentially compact: 
 
 \smallskip\noindent
 Let $\varphi_k$ be any sequence of automorphisms contained in $\Theta$.
 Then, using Montel's theorem and passing to a subsequence, the sequence 
 $\varphi_k$ converges to a morphism $\varphi:\mathcal M\rightarrow \mathcal M$.
 It remains to show that $\varphi$ is an automorphism of $\mathcal M$.
 
 \smallskip\noindent
 The underlying map $\tilde{\varphi}:M\rightarrow M$ is surjective since
 if $p\notin \tilde{\varphi}(M)$, then $\varphi\in\Delta(M,M\setminus \{p\})$ 
 and therefore $\varphi_k\in \Delta(M,M\setminus \{p\})$ for $k$ large enough
 which contradicts the assumption that $\varphi_k$ is an automorphism.
 This also implies that there is an $x\in M$ such that the differential 
 $D\tilde{\varphi}(x)$ is invertible. 
 Using Hurwitz's theorem (see e.g. \cite{Narasimhan}, p. 80) it follows
 $\det(D\tilde{\varphi}(x))\neq 0$ for all $x\in M$. Thus $\tilde{\varphi}$
 is locally biholomorphic. 
 Moreover, $\varphi$ is locally invertible due to the special form of 
 the sets $\Theta(x_i)$.
 
 \smallskip\noindent
 In order check that $\tilde{\varphi}$ is injective, let $p_1,p_2\in M$, 
 $p_1\neq p_2$, such that $q=\tilde{\varphi}(p_1)=\tilde{\varphi}(p_2)$. 
 Let $\Omega_j$, $j=1,2$, be open neighbourhoods of $p_j$ with 
 $\Omega_1\cap \Omega_2= \emptyset$. By \cite{Narasimhan}, p. 79, Proposition~5,
 there exists $k_0$ with the property that $q\in \tilde{\varphi}_k(\Omega_1)$
 and $q\in \tilde{\varphi}_k(\Omega_2)$ for all $k\geq k_0$.
 The bijectivity of the $\varphi_k$'s now yields a contradiction to 
 $\Omega_1\cap\Omega_2=\emptyset$.
\end{proof}

\begin{prop}
 The set $\Aut_{\bar 0}(\mathcal M)$ is a topological group with composition of
 automorphisms as multiplication and inversion of automorphisms as the inverse.
\end{prop}

\begin{proof}
 Let $\varphi_k$ and $\psi_l$ be two sequences of automorphisms of $\mathcal M$
 converging to $\varphi$ and $\psi$ respectively. 
 By the classical theory, $\tilde{\varphi_k}\circ\tilde{\psi_l}$ converges to 
 $\tilde{\varphi}\circ\tilde{\psi}$, and $\tilde{\varphi_k}^{-1}$ to 
 $\tilde{\varphi}^{-1}$.
 Let  $U, V, W\subseteq M$ be open subsets with $\tilde{\varphi}(V)\subseteq W$,
 $\tilde{\varphi_k}(V)\subseteq W$, $\tilde{\psi}(U)\subseteq V$, 
 $\tilde{\psi_l}(U)\subseteq V$, 
 for $k$ and $l$ sufficiently large and let $f\in \mathcal O_\mathcal M(W)$.
 Then the sequence $(\varphi_k)^*(f)\in \mathcal O_\mathcal M(V)$ converges to
 $\varphi^*(f)$ on $V$, and by Lemma \ref{le123} 
 $(\varphi_k\circ\psi_l)^*(f)=(\psi_l)^*\left((\varphi_k)^*(f)\right)$
 converges to $\psi^*(\varphi^*(f))=(\varphi\circ\psi)^*(f)$ on $U$ as
 $k,l\to \infty$ , which shows that the multiplication is continuous.
 
 \smallskip \noindent
 Consider now the inversion map 
 $\Aut_{\bar 0}(\mathcal M)\rightarrow\Aut_{\bar 0}(\mathcal M)$, 
 $\varphi\mapsto \varphi^{-1}$. Let  $\varphi_k$ be a sequence in 
 $\Aut_{\bar 0}(\mathcal M)$ converging to 
 $\varphi \in \Aut_{\bar 0}(\mathcal M)$. 
 Note that since the automorphism group $\Aut(M)$ of the underlying manifold $M$
 is a topological group, the inversion map $\Aut(M)\rightarrow \Aut(M)$ is 
 continuous. 
 For any choice of local coordinate charts on $U, V\subseteq M$ such that the 
 closure of $\tilde{\varphi}^{-1}(U)$ is contained in $V$ we can conclude: 
 Since  $\tilde{\varphi}_k^{-1}$ converges to 
 $\tilde{\varphi}^{-1}$, we have $\tilde{\varphi_k}^{-1}(U)\subseteq V$ for $k$
 sufficiently large. 
 Identify  $\mathcal O_\mathcal M(U) \cong \Gamma_{\Lambda E}(U)$, resp. 
 $\mathcal O_\mathcal M(V) \cong \Gamma_{\Lambda E}(V)$ and decompose 
 $\varphi^\ast=\varphi^\ast_0\exp(Y)$, $\varphi_k^\ast=\varphi_{k,0}^\ast
 \exp(Y_k)$ as in Section~2. Note that $\varphi_0^\ast$ is induced by an 
 automorphism $\varphi_0$ of the vector bundle $E$. We can verify by an 
 observation in local coordinates that the map 
 $\Aut_{\bar 0}(\mathcal M)\to \Aut(E)$, $\varphi \mapsto \varphi_0$, is 
 continuous. Hence, the sequence $\varphi_{k,0}$ converges to $\varphi_0$ and 
 $\varphi_{k,0}^\ast$ converges to~$\varphi_{0}^\ast$. 
 By \cite{Morimito} the inversion on  $\Aut(E)$ is continuous. 
 Therefore, $(\varphi_{k,0}^{-1})^\ast$ converges to ~$(\varphi_{0}^{-1})^\ast$.
 Due to the finiteness of the logarithm and exponential series on nilpotent 
 elements, $Y_k$ converges to $Y$. 
 Hence, $(\varphi^{-1}_k)^\ast=\exp({-Y_k})(\varphi^{\ast}_{k,0})^{-1}$ 
 converges to $\exp({-Y})(\varphi^{\ast}_{0})^{-1}=(\varphi^{\ast})^{-1}$. 
\end{proof}

\section{Non-existence of small subgroups of $\Aut_{\bar 0}(\mathcal M)$}

In this section, we prove that $\Aut_{\bar 0}(\mathcal M)$ does not 
contain small subgroups, which means that there exists an open
neighbourhood of the identity in $\Aut_{\bar 0}(\mathcal M)$ such 
that each subgroup contained in this neighbourhood consists only of the 
identity.
As a consequence, the topological group $\Aut_{\bar 0}(\mathcal M)$ carries the 
structure of a real Lie group  by a result of Yamabe (c.f. \cite{Yamabe}).

\smallskip
Before proving the non-existence of small subgroups, a few technical 
preparations are needed:
 Consider $\mathbb C^{m|n}$ and let $z_1,\ldots,z_m,\xi_1,\ldots,\xi_n$ denote 
 coordinates on $\mathbb C^{m|n}$.
 Let $U\subseteq \mathbb C^m$ be an open subset.
 For $f=\sum_\nu f_{\nu}\xi^{\nu} \in\mathcal O_{\mathbb C^{m|n}}(U)$ define 
 $$\left|\left|f\right|\right|_U =\left|\left|\sum_{\nu}f_\nu \xi^\nu\right|
 \right|_U \coloneqq\sum_\nu \left|\left|f_\nu\right|\right|_U,$$
 where $||f_\nu||_U$ denotes the supremum norm of the holomorphic 
 function $f_\nu$ on $U$.
 For any morphism 
 $\varphi:\mathcal U=(U,\mathcal O_{\mathbb C^{m|n}}|_U)
 \rightarrow \mathbb C^{m|n}$ 
 define 
 $$||\varphi||_U\coloneqq\sum_{i=1}^m ||\varphi^*(z_i)||_U
 +\sum_{j=1}^n ||\varphi^*(\xi_j)||_U.$$
 
 \begin{lemma}
 Let $\mathcal U=(U,\mathcal O_{\mathbb C^{m|n}}|_{U})$ be a superdomain in 
 $\mathbb C^{m|n}$.
 For any relatively compact open subset $U'$ of $U$ there exists
 $\varepsilon>0$ such that any morphism 
 $\psi:\mathcal U\rightarrow \mathbb C^{m|n}$ with the property
 $||\psi-\mathrm{id}||_U<\varepsilon$ is biholomorphic as a morphism from 
 $\mathcal U'=(U',\mathcal O_{\mathbb C^{m|n}}|_{U'})$ onto its image. 
 \end{lemma}

\begin{proof}
 Let $r>0$ such that the closure of the polydisc 
 $\Delta^n_r(z)=\{(w_1,\ldots, w_m)|\,|w_j-z_j|< r\}$ is contained in 
 $U$ for any $z=(z_1,\ldots, z_m)\in U'$.
 Let $v\in \mathbb C^m$ be any non-zero vector.
 Then we have $z+\zeta v\in U$ for any $z\in U'$ and 
 $\zeta$ in the closure of
 $\Delta_{\frac{r}{||v||}}(0)=\{t\in \mathbb C|\, |t|<\frac{r}{||v||}\}$.
 If for given $\varepsilon>0$ it is $||\psi-\mathrm{id}||_U<\varepsilon$ then we have in particular 
 $||\tilde\psi-\mathrm{id}||_U<\varepsilon$ for the supremum norm 
 of the underlying maps $\tilde\psi,\mathrm{id}:U\rightarrow \mathbb C^m$.
 Then, for the differential $D\tilde\psi$ of $\tilde\psi$ and any non-zero vector 
 $v\in \mathbb C^m$ and any $z \in U^\prime$ we have
 \begin{align*}
  \left|\left|D\tilde\psi (z)(v)-v\right|\right|
 &=\left|\left|\frac{d}{dt}\left(\tilde\psi(z+tv)-(z+tv)\right)\right|\right|
 =\frac{1}{2\pi}\left|\left|\int_{\partial \Delta_{\frac{r}{||v||}}(0)} 
 \frac{\tilde\psi(z+\zeta v)-(z+\zeta v)}{\zeta^2} d\zeta\right|\right|\\
 &\leq\frac{1}{2\pi}\int_{\partial\Delta_{\frac{r}{||v||}}(0)}\left|\left|
 \frac{\tilde\psi(z+\zeta v)-(z+\zeta v)}{\zeta^2} \right|\right| d\zeta\\
 &<\frac{\varepsilon ||v||}{r}.
 \end{align*}
 This implies $|| D\tilde\psi(z)-\mathrm{id}||< \frac{\varepsilon}{r}$
 with respect to the operator norm, for any $z\in U'$.
 Thus $\tilde \psi$ is locally biholomorphic on $U'$ if $\varepsilon$ is small
 enough.
 Moreover, $\varepsilon$ might now be chosen such that $\tilde\psi$ is 
 injective (see e.g. \cite{Hirsch}, Chapter 2, Lemma 1.3).
 
 \smallskip\noindent
 Let $\psi_{j,k},\psi_{j,\nu}$ be holomorphic functions on $U$
 such that 
 $\psi^*(\xi_j)=\sum_{k=1}^n \psi_{j,k}\xi_k +\sum_{||\nu||\geq 3} \psi_{j,\nu}
 \xi^\nu.$
 It is now enough to show 
 $$\det ((\psi_{j,k})_{1\leq j,k\leq n}(z))\neq 0$$
 for all $z\in U'$ and $\varepsilon$ small enough in order to prove that $\psi$
 is a biholomorphism form $\mathcal U'$ onto its image.
 This follows from the fact that we assumed,
 via $||\psi-\mathrm{id}||_U<\varepsilon$, that   
 $||\psi_{j,k}||_U< \varepsilon$ if
 $j\neq k$ and $||\psi_{j,j}-1||_U<\varepsilon$.
 \end{proof}
 
This lemma now allows us to prove that $\Aut_{\bar 0}(\mathcal M)$ contains
no small subgroups; for a similar result in the classical case 
see \cite{BMLocallyCompactGroupsDiffTransformations}, Theorem 1. 

\begin{prop}
 The topological group $\Aut_{\bar 0}(\mathcal M)$ has no small subgroups,
 i.e. there is a neighbourhood of the identity which contains no non-trivial
 subgroup.
\end{prop}

\begin{proof}
 Let $U\subset V\subset W$ be open subsets of $M$ such that 
 $U$ is relatively compact in $V$ and $V$ is relatively compact in $W$.
 Moreover, suppose that $\mathcal W=(W, \mathcal O_{\mathcal M}|_W)$
 is isomorphic to a superdomain in $\mathbb C^{m|n}$ and let
 $z_1,\ldots, z_m,\xi_1,\ldots,\xi_n$ be local coordinates on $\mathcal W$.
 By definition 
 $\Delta(\overline{V},W)=\{\varphi\in\Aut_{\bar 0}(\mathcal M)|\,\tilde{\varphi}
 (\overline{V})\subseteq W\}$
 and $\Delta(\overline{U},V)$ are open neighbourhoods of the identity
 in $\Aut_{\bar 0}(\mathcal M)$.
 Choose $\varepsilon>0$ as in the preceding lemma such that 
 any morphism $\chi:\mathcal V\rightarrow \mathbb C^{m|n}$ 
 with $||\chi-\mathrm{id}||_V<\varepsilon$
 is biholomorphic as a morphism from $\mathcal U$ onto its image.
 Let $\Omega\subseteq \Delta(\overline{V},W)\cap \Delta(\overline{U},V)$ be the
 subset whose elements
 $\varphi$ satisfy $||\varphi-\mathrm{id}||_V<\varepsilon$.
 The set $\Omega$ is open and contains the identity.
 Since $\Aut_{\bar 0}(\mathcal M)$ is locally compact by
 Lemma~\ref{lemma: locally compact}, it is enough  to show that 
 each compact subgroup $Q\subseteq \Omega$ is trivial. 
 Otherwise for non-compact $Q$, let $\Omega'$ be an open neighbourhood of the identity 
 with compact closure $\overline{\Omega}'$ which is contained in $\Omega$,
 and suppose $Q\subseteq \Omega'$. Then 
 $\overline{Q}\subseteq \overline{\Omega}'\subset \Omega$ is a compact 
 subgroup, and $Q$ is trivial if $\overline{Q}$ is trivial.
 
 Define
 a morphism $\psi:\mathcal V\rightarrow \mathbb C^{m|n}$
 by setting 
 $$\psi^*(z_i)=\int_Q q^*(z_i)\,dq\ \ \text{ and }\ \ 
 \psi^*(\xi_j)=\int_Q q^*(\xi_j)\,dq,$$
 where the integral is taken with respect to the normalized Haar measure on $Q$.
 This yields a holomorphic morphism 
 $\psi:\mathcal V\rightarrow \mathbb C^{m|n}$ since each $q\in Q$ defines a
 holomorphic morphism 
 $\mathcal V\rightarrow\mathcal W\subseteq \mathbb C^{m|n}$.
 Its underlying map is
 $\tilde{\psi}(z)=\int_Q \tilde q (z) \,dq$.
 The morphism $\psi$ satisfies
 $$||\psi^*(z_i)-z_i||_V
 =\left|\left|\int_Q (q^*(z_i)-z_i)\,dq\right|\right|_V
 \leq \int_Q ||q^*(z_i)-z_i||_V\, dq
 $$
 and similarly
 $$||\psi^*(\xi_j)-\xi_j||_V
 \leq \int_Q ||q^*(\xi_j)-\xi_j||_V \, dq.$$
 Consequently, we have
 \begin{align*}
  ||\psi-\mathrm{id}||_V
 &= \sum_{i=1}^m ||\psi^*(z_i)-z_i||_V+\sum_{j=1}^n || \psi^*(\xi_j)-\xi_j||_V\\
 &\leq \int_Q\left(\sum_{i=1}^m||q^*(z_i)-z_i||_V\, +
 \sum_{j=1}^n ||q^*(\xi_j)-\xi_j||_V \right) \, dq \\
 &=\int_Q ||q-\mathrm{id}||_V\, dq
 < \varepsilon.
 \end{align*}

 Thus by the preceding lemma, $\psi|_U$ is a biholomorphic morphism onto its image.
 Furthermore, on $U$ we have
 $\psi\circ q'=\psi$ for any $q'\in Q$
 since 
 \begin{align*}
  (\psi\circ q')^*(z_i)
  &=(q')^*(\psi^*(z_i))
 =(q')^*\left(\int_Q q^*(z_i)\,dq\right)
 =\int_Q (q')^*(q^*(z_i))\,dq\\
 &=\int_Q(q\circ q')^*(z_i)\,dq
 =\int_Q q^*(z_i)\,dq
 =\psi^*(z_i)
 \end{align*}
 due to the invariance of the Haar measure,
 and also
 $$(\psi\circ q')^*(\xi_j)=\psi^*(\xi_j).$$
 The equality $\psi\circ q'=\psi$ on $U$ implies 
 $q'|_U=\mathrm{id}_\mathcal U$ because of the invertibility of $\psi$.
 By the identity principle it follows that $q'=\mathrm{id}_{\mathcal M}$
 if $M$ is connected, and hence $Q=\{\mathrm{id}_{\mathcal M}\}$.
 
 In general, $M$ has only finitely many connected components since $M$ is 
 compact. Therefore, a repetition of the preceding argument yields the existence
 of a neighbourhood of the identity of $\Aut_{\bar 0}(\mathcal M)$ without 
 any non-trivial subgroups.
\end{proof}

By Theorem~3 in \cite{Yamabe}, the preceding proposition implies the following:
\begin{cor}\label{cora}
 The topological group $Aut_{\bar 0}(\mathcal M)$ can be endowed with the
 structure of a real Lie group.
\end{cor}


\section{One-parameter subgroups of $\Aut_{\bar 0}(\mathcal M)$}

In order to obtain results on the regularity of the action of 
$\Aut_{\bar 0}(\mathcal M)$ on the compact complex supermanifold $\mathcal M$
and to characterize the Lie algebra of $\Aut_{\bar 0}(\mathcal M)$,
we study continuous one-parameter subgroups of $\Aut_{\bar 0}(\mathcal M)$.
Each continuous 
one-parameter subgroup $\mathbb R\rightarrow
\Aut_{\bar 0}(\mathcal M)$ is an analytic map between the Lie groups
$\mathbb R$ and $\Aut_{\bar 0}(\mathcal M)$.

We prove that the action of each continuous one-parameter
subgroup of $\Aut_{\bar 0}(\mathcal M)$ on $\mathcal M$ is analytic and
induces an even holomorphic super vector field on $\mathcal M$. 
Consequently, the Lie algebra of $\Aut_{\bar 0}(\mathcal M)$ may 
be identified with the Lie algebra $\mathrm{Vec}_{\bar 0}(\mathcal M)$ of 
even holomorphic super vector fields on $\mathcal M$, and
$\Aut_{\bar 0}(\mathcal M)$ carries the structure of a complex Lie group 
whose action on the supermanifold $\mathcal M$ is holomorphic.

\begin{defi}
 A continuous one-parameter subgroup of automorphisms of $\mathcal M$ is a 
 family of automorphisms $\varphi_t:\mathcal M\rightarrow\mathcal M$, 
 $t\in \mathbb R$ such that the map
 $\mathbb R\rightarrow \Aut_{\bar 0}(\mathcal M)$,
 $t\mapsto \varphi_t$, is a continuous group homomorphism. 
\end{defi}

\begin{rmk}\label{rmk: cont. 1psg}
 Let $\varphi_t:\mathcal M\rightarrow\mathcal M$, $t\in\mathbb R$, be a family 
 of automorphisms satisfying $\varphi_{s+t}=\varphi_s\circ\varphi_t$ for
 all $s,t\in \mathbb R$, and such that
 $\tilde\varphi:\mathbb R\times M\rightarrow M$, $\tilde\varphi(t,p)=
 \tilde{\varphi}_t(p)$ is continuous.
 Then $\varphi_t$ is a continuous one-parameter subgroup if and only if
 the following condition is satisfied:
 Let $U, V\subset M$ be open subsets, and $[a,b]\subset \mathbb R$
 such that $\tilde\varphi([a,b]\times U)\subseteq V$.
 Assume moreover that there are local coordinates
 $z_1,\ldots, z_m,\xi_1,\ldots,\xi_n$  for $\mathcal M$ on $U$. 
 Then for any $f\in\mathcal O_{\mathcal M}(V)$ there are continuous
 functions $f_\nu:[a,b]\times U\rightarrow \mathbb C$ with 
 $(f_\nu)_t=f_\nu(t,\cdot)\in \mathcal O_{\mathcal M}(U)$ for fixed 
 $t\in [a,b]$ such
 that $$(\varphi_t)^*(f)=\sum_\nu f_\nu (t,z)\xi^\nu.$$
 We say that the action of the one-parameter subgroup $\varphi$ on 
 $\mathcal M$ is analytic if each 
 $f_\nu (t,z)$ is analytic in both components.
\end{rmk}

\begin{prop}\label{prop: 1-psg}
 Let $\varphi$ be continuous one-parameter subgroup of automorphisms on 
 $\mathcal M$. Then the action of $\varphi$ on $\mathcal M$ is analytic.
\end{prop}

\begin{rmk}
 Defining a continuous one-parameter subgroup as in
 Remark~\ref{rmk: cont. 1psg}, the statement of Proposition~\ref{prop: 1-psg} 
 also holds true for complex supermanifolds $\mathcal M$ with non-compact 
 underlying manifold $M$ as compactness of $M$ is not needed for the proof.
\end{rmk}

For the proof of the proposition the following technical lemma is needed:

\begin{lemma}\label{lemma: integration}
 Let $U\subseteq V\subseteq \mathbb C^m$ be open subsets, $p\in U$,
 $\Omega\subseteq \mathbb R$ an open connected neighbourhood of $0$, and let
 $\alpha:\Omega\times U\rightarrow V$ be a continuous map 
 satisfying $\alpha(t,z)=\alpha(t+s,z)-f(t,s,z)$ for some continuous function 
 $f$ which is analytic in $(t,z)$ and small $s,t$, and $z$ near $p$. 
 If $\alpha$ is holomorphic in the second component, then it is analytic 
 on a neighbourhood of $(0,p)$. 
\end{lemma}

\begin{proof}
 For small $t$, $h>0$, $z$ near $p$, we have
 \begin{align*}
  h\cdot\alpha(t,z)&=\int_0^h\alpha(t+s,z)ds - \int_0^h f(t,s,z) ds\\
  &=\int_t^{h+t}\alpha(s,z)ds -  \int_0^h \alpha(s,z)ds 
  -\int_0^h (f(t,s,z)-\alpha(s,z))ds\\
  &=\int_h^{h+t}\alpha(s,z)ds-\int_0^t\alpha(s,z)ds
  -\int_0^h(f(t,s,z)-\alpha(s,z))ds\\
  &=\int_0^t(\alpha(s+h,z)-\alpha(s,z))ds
    -\int_0^h(f(t,s,z)-\alpha(s,z))ds\\
  &=\int_0^t f(s,h,z)ds -\int_0^h(f(t,s,z)-\alpha(s,z))ds.
  \end{align*}
The assumption that $f$ is a continuous function which is analytic in the first 
and third component therefore implies that $\alpha$ is analytic.
\end{proof}

\begin{proof}[Proof of Proposition \ref{prop: 1-psg}]
 Due to the action property $\varphi_{s+t}=\varphi_s\circ\varphi_t$ it
 is enough to show the statement for the restriction of $\varphi$ to
 $(-\varepsilon,\varepsilon)\times\mathcal M$ for some $\varepsilon >0$.
 Let $U,V\subseteq M$ be open subsets such that $U$ is relatively compact in 
 $V$, and such that there are local coordinates $z_1,\ldots, z_m,\xi_1,\ldots,
 \xi_n$ on $V$ for $\mathcal M$.
 Choose $\varepsilon>0$ such that
 $\tilde\varphi_t(U)\subseteq V$ for any $t\in (-\varepsilon,\varepsilon)$.
 Let $\alpha_{i,\nu}$, $\beta_{j,\nu}$ be continuous functions on 
 $(-\varepsilon,\varepsilon)\times U$ with
 $(\varphi_t)^*(z_i)=\sum_{|\nu|=0} \alpha_{i,\nu}(t,z) \xi^\nu$
 and $(\varphi_t)^*(\xi_j)=\sum_{|\nu|=1} \beta_{j,\nu}(t,z) \xi^\nu$,
 where $|\nu|=|(\nu_1,\ldots,\nu_n)|=(\nu_1+\ldots +\nu_n)\!\!\mod 2 \in\mathbb Z_2$.
 We have to show that $\alpha$ and $\beta$ are analytic in $(t,z)$.
 The induced map 
 $\psi':(-\varepsilon,\varepsilon)\times U\times\mathbb C^n\rightarrow
 V\times \mathbb C^n$ on the underlying vector bundle is given by 
 $$\left(t,\left(\begin{matrix}
          z_1\\ \vdots \\ z_m \\ v_1\\ \vdots \\ v_n
         \end{matrix}\right)\right)
         \mapsto
         \left(\begin{matrix}
          \alpha_{1,0}(t,z)\\ \vdots \\ \alpha_{m,0}(t, z)\\
  \sum_{k=1}^n \beta_{1,k}(t,z)v_k\\ \vdots \\ \sum_{k=1}^n \beta_{n,k}(t,z)v_k
         \end{matrix}\right),$$
where $\beta_{j,k}=\beta_{j,e_k}$ if $e_k=(0,\ldots,0,1,0,\ldots, 0)$
denotes the $k$-th unit vector.
The map $\psi'$ is a local continuous one-parameter subgroup on 
$U\times \mathbb C^n$ 
because $\varphi$ is a continuous one-parameter subgroup.
By a result of Bochner and Montgomery the map $\psi'$ is analytic in $(t,z,v)$
(see \cite{BMGroupsDiffAnalyticTransformations}, Theorem 4).
Hence, the map
$\psi:(-\varepsilon,\varepsilon)\times \mathcal U\rightarrow \mathcal V$
given by 
$(\psi_t)^*(z_i)=\alpha_i(t,z)$, 
$(\psi_t)^*(\xi_j)=\sum_{k=1}^n \beta_{j,k}(t,z)\xi_k$
is analytic.
Let $X$ be the local vector field on $\mathcal U$ induced by $\psi$, i.e.
$$X(f)=\left.\frac{\partial}{\partial t}\right|_0 (\psi_t)^*(f).$$
We may assume that $X$ is non-degenerate, i.e. the evaluation of $X$ in $p$,
$X(p)$, does not vanish for all $p\in U$.
Otherwise, consider, instead of $\varphi$, the diagonal action on
$\mathbb C\times \mathcal M$ acting by addition of $t$ in the first component
and $\varphi_t$ in the second, and note that this action is analytic precisely 
if $\varphi$ is analytic.
For the differential $d\psi$ of $\psi$ in $(0,p)$ we have
$$d\psi\left(\left.\frac{\partial}{\partial t}\right|_{(0,p)}\right)
=\left.\frac{\partial}{\partial t}\right|_{(0,p)}
\circ \psi^*
=X(p)\neq 0.$$
Therefore, the restricted map $\psi|_{(-\varepsilon,\varepsilon)\times\{p\}}$
is an immersion and its image $\psi((-\varepsilon,\varepsilon)\times\{p\})$
is a subsupermanifold of $\mathcal V$.
Let $\mathcal S$ be a subsupermanifold of $\mathcal U$ transversal to 
$\psi((-\varepsilon,\varepsilon)\times\{p\})$ in $p$.
The map $\psi|_{(-\varepsilon,\varepsilon)\times \mathcal S}$
is a submersion in $(0,p)$ since
$d\psi (T_{(0,p)}(-\varepsilon,\varepsilon)\times\{p\}))
= T_p \psi((-\varepsilon,\varepsilon)\times\{p\})$ and 
$d\psi (T_{(0,p)}\{0\}\times\mathcal S)=T_p\mathcal S$ because
$\psi|_{\{0\}\times \mathcal U}=\mathrm{id}$.
Hence $\chi:=\psi|_{(-\varepsilon,\varepsilon)\times \mathcal S}$
is locally invertible around $(0,p)$, 
and thus invertible as a map onto its image after possibly shrinking $U$ and
$\varepsilon$,
and $$\chi_*\left(\frac{\partial}{\partial t}\right)
=(\chi^{-1})^*\circ \frac{\partial}{\partial t}\circ \chi^*
=(\chi^{-1})^*\circ \chi^*\circ X=X.$$
Therefore, after defining new coordinates 
$w_1,\ldots, w_m,\theta_1,\ldots, \theta_n$ for $\mathcal M$ on $U$ via
$\chi$, we have $X=\frac{\partial}{\partial w_1}$ and $(\varphi_t)^*$ is of the 
form
\begin{align*}
 (\varphi_t)^*(w_1)&=w_1+t+\sum_{|\nu|=0, \nu\neq 0} \alpha_{1,\nu}(t,w)\theta^\nu,\\
 (\varphi_t)^*(w_i)&=w_i+\sum_{|\nu|=0,\nu\neq 0} \alpha_{i,\nu}(t,w)\theta^\nu
 \ \ \ \text{ for } i\neq 1,\\
 (\varphi_t)^*(\theta_j)&=\theta_j+\sum_{|\nu|=1,||\nu||\neq 1} 
 \beta_{j,\nu}(t,w)\theta^\nu,
\end{align*}
for appropriate $\alpha_{i,\nu}$, $\beta_{j,\nu}$,
where $||\nu||=||(\nu_1,\ldots,\nu_n)||=\nu_1+\ldots +\nu_n$.

For small $s$ and $t$ we have 
\begin{align}
 \varphi_t^*&\left(\varphi_s^*(w_i)\right)=\varphi_t^*\left(w_i+\delta_{1,i}s+
\sum_{|\nu|=0,||\nu||\neq 0} \alpha_{i,\nu}(s,w)\theta^\nu\right)\nonumber\\ 
&= w_i+\delta_{i,1}(t+s)
+\sum_{|\nu|=0,||\nu||\neq 0} \alpha_{i,\nu}(t,w)\theta^\nu
+\sum_{|\nu|=0,||\nu||\neq 0} \varphi_t^*(\alpha_{i,\nu}(s,w)\theta^\nu).
\end{align}

Let $f_{i,\nu}(t,s,w)$ be such that
\begin{align}\sum_{|\nu|=0,||\nu||\neq 0} \varphi_t^*(\alpha_{i,\nu}(s,w)\theta^\nu)
=\sum_{|\nu|=0,||\nu||\neq 0} f_{i,\nu}(t,s,w) \theta^\nu.\end{align}
For fixed $\nu_0$ the coefficient of $\theta^{\nu_0}$, $f_{i,{\nu_0}}(t,s,w)$,
depends only on $\alpha_{i,\nu_0}(s,w+t e_1)$, 
$\beta_{j,\mu}(t,w)$ for $\mu$ with $||\mu||\leq||\nu_0||-1$,
and $\alpha_{j,\nu}(t,w)$ and its partial derivatives in
the second component for $\nu$ with $||\nu||\leq||\nu_0||-2$.
This can be shown by a calculation using the special form of 
$\varphi_t^*(w_j)$ and $\varphi_t^*(\theta_j)$ and general properties of the 
pullback of a morphism of supermanifolds.
Assume now that the analyticity near $(0,p)$ of $\alpha_{i,\nu}$, $\beta_{j,\mu}$ 
is shown for $||\nu||, ||\mu||< 2k$ and all $i,j$. Let $\nu_0$ be such that
$||\nu_0||=2k$. Then $f_{i,\nu_0}(t,s,w)$ is a continuous function which is 
analytic in $(t,w)$ near $(0,p)$ for fixed s.
Since $\varphi_t^*(\varphi_s^*(w_i))
=\varphi_{t+s}^*(w_i)$, using $(1)$ and $(2)$ we get 
$$\alpha_{i,\nu_0}(t,w)+f_{i,\nu_0}(t,s,w)
=\alpha_{i,\nu_0}(t+s,w),$$
and thus $\alpha_{i,\nu_0}(t,w)$ is analytic near $(0,p)$ by 
Lemma~\ref{lemma: integration}. Similarly, it can be shown that 
$\beta_{j,\mu_0}$ is analytic for $||\mu_0||=2k+1$ 
if $\alpha_{i,\nu}$, $\beta_{j,\mu}$ for $||\nu||$, $||\mu||< 2k+1$.
\end{proof}

\begin{cor}\label{cor: Lie algebra of Aut_0}
 The Lie algebra of $\Aut_{\bar 0}(\mathcal M)$ is isomorphic to the Lie 
 algebra $\mathrm{Vec}_{\bar 0}(\mathcal M)$ of even super vector fields on 
 $\mathcal M$, and $\Aut_{\bar 0}(\mathcal M)$ is a complex Lie group.
\end{cor}

\begin{proof}
 If $\gamma:\mathbb R\rightarrow \Aut_{\bar 0}(\mathcal M)$, $t\mapsto \gamma_t$
 is a continuous one-parameter subgroup, then by 
 Proposition~\ref{prop: 1-psg} the action of $\varphi$ on $\mathcal M$ is
 analytic.
 Therefore, $\gamma$ induces an even holomorphic super vector field~$X(\gamma)$
 on $\mathcal M$ by setting 
 $$X(\gamma)=\left.\frac{\partial}{\partial t}\right|_0 (\gamma_t)^*,$$
 and $\gamma $ is the flow map of $X(\gamma)$.
  On the other hand, each $X\in \mathrm{Vec}_{\bar 0}(\mathcal M)$
 is globally integrable since $M$ is compact (c.f \cite{GarnierWurzbacher}, 
 Theorem~5.4). Its flow defines a one-parameter subgroup 
 $\gamma^X$ of $\Aut_{\bar 0}(\mathcal M)$, which is continuous.
 This yields an isomorphism of Lie algebras 
 $$\mathrm{Lie}(\Aut_{\bar 0}(\mathcal M))
 \rightarrow \mathrm{Vec}_{\bar 0}(\mathcal M).$$
 Consequently, we have $\mathrm{Lie}(\Aut_{\bar 0}(\mathcal M))
\cong \mathrm{Vec}_{\bar 0}(\mathcal M)$
 and since $\mathrm{Vec}_{\bar 0}(\mathcal M)$ is a complex Lie algebra,
 $\Aut_{\bar 0}(\mathcal M)$ carries the structure of a complex Lie group.
\end{proof}

The Lie group $\Aut_{\bar 0}(\mathcal M)$ naturally acts on $\mathcal M$; this 
action $\psi:\Aut_{\bar 0}(\mathcal M)\times\mathcal M\rightarrow \mathcal M$ is
given by $\mathrm{ev}_g\circ\psi^*=g^*$ where $\mathrm{ev}_g $ denotes the 
evaluation in $g\in \Aut_{\bar 0}(\mathcal M)$ in the first component.

\begin{cor} \label{cory}
 The natural action of $\Aut_{\bar 0}(\mathcal M)$ on $\mathcal M$ defines 
 a holomorphic morphism of supermanifolds 
 $\Aut_{\bar 0}(\mathcal M)\times\mathcal M \rightarrow\mathcal M$.
\end{cor}

\begin{proof}
 Since the action of each continuous one-parameter subgroup of
$ \Aut_{\bar 0}(\mathcal M)$ on $\mathcal M$ is holomorphic by 
the preceding considerations, and each $g\in \Aut_{\bar 0}(\mathcal M)$
is a biholomorphic morphism $g:\mathcal M\rightarrow\mathcal M$,
the action~$\Phi$ is a holomorphic. 
\end{proof}

If a Lie supergroup $\mathcal G$ (with Lie superalgebra $\mathfrak g$
of right-invariant super vector fields)
acts on a supermanifold 
$\mathcal M$ via $\psi:\mathcal G\times \mathcal M\rightarrow
\mathcal M$, this action $\psi$ induces an infinitesimal action 
$d\psi:\mathfrak g\rightarrow\mathrm{Vec}(\mathcal M)$ defined by 
$d\psi(X)=(X(e)\otimes \mathrm{id}_\mathcal M^*)\circ \psi^*$
for any $X\in\mathfrak g$, 
where $X\otimes \mathrm{id}_\mathcal M^*$ denotes the canonical 
extension of the vector field $X$ on
$\mathcal G$ to a vector field on $\mathcal G\times \mathcal M$,
and $(X(e)\otimes\mathrm{id}_\mathcal M^*)$ is its evaluation in the neutral 
element $e$ of $\mathcal G$.

\begin{cor}\label{new-cor}
Identifying the Lie algebra of $\Aut_{\bar 0}(\mathcal M)$ with
$\mathrm{Vec}_{\bar 0}(\mathcal M)$ as in Corollary 
\ref{cor: Lie algebra of Aut_0}, the induced infinitesimal action of the action 
$\psi:\Aut_{\bar 0}(\mathcal M)\times\mathcal M\rightarrow\mathcal M$ in 
Corollary \ref{cory} is the inclusion
$\mathrm{Vec}_{\bar 0}(\mathcal M)\hookrightarrow\mathrm{Vec}(\mathcal M)$.
\end{cor}

\section{The Lie superalgebra of vector fields }
In this section, we prove that the Lie superalgebra $\mathrm{Vec}(\mathcal M)$
of holomorphic super vector
fields on a compact complex supermanifold $\mathcal M$ is finite-dimensional. 

First, we prove that $\mathrm{Vec}(\mathcal M)$ is finite-dimensional if
$\mathcal M$ is a split supermanifold
using that its tangent sheaf $\mathcal T_{\mathcal M}$ is a coherent sheaf of 
$\mathcal O_M$-modules, where $\mathcal O_M$ denotes again the sheaf of
holomorphic functions on the underlying manifold $M$.
Then the statement in the general case is deduced
using a filtration of the tangent sheaf.

Remark that since $\Aut_{\bar 0}(\mathcal M)$ is a complex Lie group with Lie
algebra isomorphic to the Lie algebra $\mathrm{Vec}_{\bar 0}(\mathcal M)$ of 
even holomorphic super vector fields on $\mathcal M$ 
(see Corollary~\ref{cor: Lie algebra of Aut_0}), we
already know that the even part of $\mathrm{Vec}(\mathcal M)
=\mathrm{Vec}_{\bar 0}(\mathcal M)\oplus \mathrm{Vec}_{\bar 1}(\mathcal M)$
is finite-dimensional.

\begin{lemma}
 Let $\mathcal M$ be a split complex supermanifold.
 Then its tangent sheaf $\mathcal T_\mathcal M$ is a coherent sheaf of
 $\mathcal O_M$-modules.
\end{lemma}

\begin{proof}
 Since $\mathcal M$ is split, its structure sheaf $\mathcal O_\mathcal M$ is
 isomorphic to $\bigwedge \mathcal E$ as an $\mathcal O_M$-module, where
 $\mathcal E$ is the sheaf of
 sections of a holomorphic vector bundle on the underlying manifold $M$. 
 Thus, the structure sheaf $\mathcal O_\mathcal M$, and hence also the tangent
 sheaf~$\mathcal T_\mathcal M$, carry the structure of a 
 sheaf of $\mathcal O_M$-modules.
   Let $U\subset M$ be an open subset such that there exist even coordinates 
 $z_1,\ldots, z_m$ and odd 
 coordinates $\xi_1,\ldots, \xi_n$. Any derivation
 $D\in \mathcal T_\mathcal M(U)$ on $U$ can uniquely be written as 
 $$D=\sum_{\nu\in (\mathbb Z_2)^n}\left(\sum_{i=1}^m f_{i,\nu} (z)
 \xi^\nu\frac{\partial}{\partial z_i}
 +\sum_{j=1}^n  g_{j,\nu}(z)\xi^\nu\frac{\partial}{\partial \xi_j}\right)$$
 where $f_{i,\nu}$, $g_{j,\nu}$ are holomorphic functions on $U$.
 Therefore, the restricted sheaf $\mathcal T_\mathcal M|_U$ is isomorphic to 
 $(\mathcal O_M|_U)^{2^n (m+n)}$ and $\mathcal T_\mathcal M$ is coherent over 
 $\mathcal O_M$.
 \end{proof}

\begin{prop}\label{vec}
 The Lie superalgebra $\mathrm{Vec}(\mathcal M)$ of holomorphic 
 super vector fields on a compact complex supermanifold $\mathcal M$
 is finite-dimensional.
\end{prop}

\begin{proof}
 First, assume that $\mathcal M$ is split. Then the tangent sheaf
 $\mathcal T_\mathcal M$ is a coherent sheaf of $\mathcal O_M$-modules.
 Thus, the 
 space of global sections of $\mathcal T_\mathcal M$, $\mathrm{Vec}(\mathcal M)
 =\mathcal T_\mathcal M(M)$, is finite-dimensional 
 since $M$ is compact (c.f. \cite{CartanSerre}). 
 
 Now, let $\mathcal M$ be an arbitrary compact complex supermanifold.
 We associate the split complex supermanifold 
 $\text{gr}\,\mathcal M=(M, \text{gr}\, \mathcal O_\mathcal M)$
 as described in Section~2. Let $\mathcal I_\mathcal M$
 denote as before the subsheaf of ideal in $\mathcal O_\mathcal M$
 generated by the odd elements.
 Define the filtration of sheaves of Lie superalgebras
 $$\mathcal T_\mathcal M=:(\mathcal T_\mathcal M)_{(-1)}\subset 
 (\mathcal T_\mathcal M)_{(0)}\subset
 (\mathcal T_\mathcal M)_{(1)}\subset \ldots \subset
 (\mathcal T_\mathcal M)_{(n+1)}=0$$
 of the tangent sheaf $\mathcal T_\mathcal M$
 by setting 
 $$(\mathcal T_\mathcal M)_{(k)}=\{D\in \mathcal T_\mathcal M|\,
 D(\mathcal O_\mathcal M)\subset (\mathcal I_\mathcal M)^k,
 \,D(\mathcal I_\mathcal M)\subset (\mathcal I_\mathcal M)^{k+1}\}$$
 for $k\geq 0$.
 Moreover, define $\text{gr}_k(\mathcal T_\mathcal M)=
 (\mathcal T_\mathcal M)_{(k)}/(\mathcal T_\mathcal M)_{(k+1)}$
 and set 
 $$\text{gr}(\mathcal T_\mathcal M)=\bigoplus_{k\geq -1}
 \text{gr}_k(\mathcal T_\mathcal M).$$
 By \cite{Onishchik}, Proposition~1, the sheaf 
 $\text{gr}(\mathcal T_\mathcal M)$
 is isomorphic to the tangent sheaf of the associated split supermanifold
 $\text{gr}\, \mathcal M$.
 By the preceding considerations, the space of holomorphic super vector fields 
 on $\text{gr}\,\mathcal M$, 
 $$\text{Vec}(\text{gr}\,\mathcal M)=\text{gr}(\mathcal T_\mathcal M)(M)=
 \bigoplus_{k\geq -1} \text{gr}_k(\mathcal T_\mathcal M)(M),$$
 is of finite dimension.
 The projection onto the quotient yields 
 $$\dim (\mathcal T_\mathcal M)_{(k)}(M)-\dim(\mathcal T_\mathcal M)_{(k+1)}(M)
 \leq \dim (\text{gr}_k(\mathcal T_\mathcal M)(M))$$ 
 and 
 $\dim (\mathcal T_\mathcal M)_{(n)}(M)
 =\dim (\text{gr}_n(\mathcal T_\mathcal M)(M))$ 
 and hence by induction
 $$\dim (\mathcal T_\mathcal M)_{(k)}(M)
 \leq \sum_{j\geq k}\dim(\text{gr}_j(\mathcal T_\mathcal M)(M)),$$
 which gives 
 $$\dim (\mathcal T_\mathcal M(M))
 =\dim \left((\mathcal T_\mathcal M)_{(-1)}(M)\right)
 \leq \dim \left(\text{gr}(\mathcal T_\mathcal M)(M)\right).$$
 In particular, $\dim (\mathcal T_\mathcal M(M))$ is finite.
\end{proof}

\begin{rmk}
 The proof of the preceding proposition also shows the following
 inequality:
 $$\dim (\text{Vec}(\mathcal M))\leq \dim(\text{Vec}(\text{gr}\,\mathcal M))$$
\end{rmk}

\section{The automorphism group}

In this section, the automorphism group of a compact complex supermanifold 
is defined. This is done via the formalism of Harish-Chandra pairs for complex
Lie supergroups (c.f. \cite{Vishnyakova}).
The underlying classical Lie group is $\Aut_{\bar 0}(\mathcal M)$ and
the Lie superalgebra is $\mathrm{Vec}(\mathcal M)$, the Lie superalgebra of 
super vector fields on $\mathcal M$.
Moreover, we prove that the automorphism group satisfies a universal property. 

\smallskip

Consider the representation $\alpha$ of $\Aut_{\bar 0}(\mathcal M)$ on 
$\mathrm{Vec}(\mathcal M)$ given by 
$$\alpha(g)(X)=g_*(X)=(g^{-1})^*\circ X\circ g^*\ \ \ \text{ for }
\ \ \ g\in\Aut_{\bar 0}(\mathcal M), \, X\in \mathrm{Vec}(\mathcal M).$$
This representation $\alpha$ preserves the parity on $\mathrm{Vec}(\mathcal M)$,
and its restriction to $\mathrm{Vec}_{\bar 0}(\mathcal M)$ coincides
with the adjoint action of $\Aut_{\bar 0}(\mathcal M)$ on its Lie 
algebra $\mathrm{Lie}(\Aut_{\bar 0}(\mathcal M))\cong \mathrm{Vec}_{\bar 0}
(\mathcal M)$.
Moreover, the differential $(d\alpha)_{\mathrm{id}}$ at the identity 
$\mathrm{id}\in\Aut_{\bar 0}(\mathcal M)$ is the adjoint representation  
of $\mathrm{Vec}_{\bar 0}(\mathcal M)$ on $\mathrm{Vec}(\mathcal M)$:\\
Let $X$ and $Y$ be super vector fields on $\mathcal M$. Assume that $X$ is 
even and let $\varphi^X$ denote the corresponding one-parameter subgroup.
Then we have 
$$(d\alpha)_{\mathrm{id}}(X)(Y)=\left.\frac{\partial}
{\partial t}\right|_0 (\varphi^X_t)_*(Y)=[X,Y];$$
see e.g. \cite{Bergner}, Corollary~3.8.
Therefore, the pair $(\Aut_{\bar 0}(\mathcal M),\mathrm{Vec}(\mathcal M))$
together with the representation $\alpha$ is a complex Harish-Chandra pair, 
and using the equivalence between the category of complex Harish-Chandra 
pairs and complex Lie supergroups (c.f. \cite{Vishnyakova}, $\S$~2), we can
define the automorphism group of a compact complex supermanifold $\mathcal M$
as follows:

\begin{defi}\label{defi: automorphism group}
 Define the automorphism group $\Aut(\mathcal M)$ of a compact
 complex supermanifold to be the unique complex Lie supergroup associated
 to the Harish-Chandra pair 
 $(\Aut_{\bar 0}(\mathcal M), \mathrm{Vec}(\mathcal M))$ with adjoint representation
 $\alpha$.
\end{defi}

Since the action $\psi:\Aut_{\bar 0}(\mathcal M)\times\mathcal M\rightarrow
\mathcal M$ induces the inclusion $\mathrm{Vec}_{\bar 0}(\mathcal M)
\hookrightarrow \mathrm{Vec}(\mathcal M)$ as infinitesimal action (see Corollary \ref{new-cor}),
there exists a Lie supergroup action $\Psi:\Aut(\mathcal M)\times\mathcal M
\rightarrow \mathcal M$ with the identity $\mathrm{Vec}(\mathcal M)
\rightarrow\mathrm{Vec}(\mathcal M)$ as induced infinitesimal action 
and $\Psi|_{\Aut_{\bar 0}(\mathcal M)\times\mathcal M}=\psi$ (c.f. Theorem~5.35
in \cite{Bergner}).

The automorphism group together with $\Psi$ satisfies a universal property:

\begin{thm}\label{thm: universal property}
 Let $\mathcal G$ be a complex Lie supergroup with a holomorphic action
 $\Psi_{\mathcal G}:\mathcal G\times \mathcal M\rightarrow \mathcal M$.
 Then there is a unique morphism $\sigma:\mathcal G\rightarrow 
 \Aut(\mathcal M)$ of Lie supergroups 
 such that the diagram
$$\xymatrix{\mathcal G\times\mathcal M
\ar[dr]_-{\sigma\times \mathrm{id}_\mathcal M }
 \ar[rr]^{\Psi_\mathcal G}&&\mathcal M\\
 &\Aut(\mathcal M)\times \mathcal M \ar[ru]_-{\Psi}&&
}$$
is commutative.
\end{thm}

\begin{proof} 
 Let $G$ be the underlying Lie group of $\mathcal G$.
 For each $g\in G$, we have a morphism $\Psi_\mathcal G(g):\mathcal M\rightarrow 
 \mathcal M$ by setting 
 $(\Psi_\mathcal G(g))^*=\mathrm{ev}_g\circ (\Psi_\mathcal G)^*$. This morphism
 $\Psi_\mathcal G(g)$ is an automorphism of $\mathcal M$ with inverse 
 $\Psi_\mathcal G(g^{-1})$ and gives rise 
 to a group homomorphism $\tilde\sigma:G\rightarrow \Aut_{\bar 0}(\mathcal M)$,
 $g\mapsto \Psi_\mathcal G(g)$.
 
 Let $\mathfrak g$ denote the Lie superalgebra (of right-invariant super vector 
 fields) of $\mathcal G$, and
 $d\Psi_\mathcal G:\mathfrak g\rightarrow\mathrm{Vec}(\mathcal M)$ the 
 infinitesimal action induced by $\Psi_\mathcal G$. The restriction 
 of $d\Psi_\mathcal G$ to the even part $\mathfrak{g}_{\bar 0}=\mathrm{Lie}(G)$
 of $\mathfrak g$ coincides with the differential $(d\tilde \sigma)_e$ of
 $\tilde\sigma$ at the identity $e\in G$.
 
 Moreover, if $\alpha_\mathcal G$ denotes the adjoint action of $G$ on 
 $\mathfrak g$, and $\alpha$ denotes, as before, the adjoint action of 
 $\Aut_{\bar 0}(\mathcal M)$ on $\mathrm{Vec}(\mathcal M)$, 
 we have 
 \begin{align*}
d\Psi_\mathcal G(\alpha_\mathcal G(g)(X))
 &=(\Psi_\mathcal G(g^{-1}))^*\circ d\Psi_\mathcal G(X)
 \circ (\Psi_\mathcal G(g))^*
 =(\tilde\sigma(g^{-1}))^*\circ d\Psi_\mathcal G(X)\circ (\tilde\sigma(g))^*\\
 &=\alpha(\tilde\sigma(g))( d\Psi_\mathcal G(X))
 \end{align*}
 for any $g\in G$, $X\in\mathfrak g$.
 Using the correspondence between Lie supergroups and Harish-Chandra pairs, 
 it follows that there is a unique morphism $\sigma:\mathcal G
 \rightarrow\Aut(\mathcal M)$ of Lie supergroups 
 with underlying map $\tilde\sigma$ and derivative $d\Psi_\mathcal G:
 \mathfrak g\rightarrow \mathrm{Vec}(\mathcal M)$ 
 (see e.g. \cite{Vishnyakova}, $\S$~2), and $\sigma$ satisfies
 $\Psi\circ(\sigma\times\mathrm{id}_\mathcal M)=\Psi_\mathcal G$.
 
 The uniqueness of $\sigma$ follows from the fact that 
 each morphism $\tau:\mathcal G\rightarrow\Aut(\mathcal M)$ of 
 Lie supergroups fulfilling the same properties as $\sigma$ necessarily
 induces the map
 $d\Psi_\mathcal G: \mathfrak g\rightarrow \mathrm{Vec}(\mathcal M)$
 on the level of Lie superalgebras 
 and its underlying map $\tilde\tau$ has to satisfy 
 $\tilde\tau(g)=\Psi_\mathcal G(g)=\tilde\sigma(g)$.
\end{proof}

\begin{rmk}
 Since the morphism $\sigma$ in Theorem~\ref{thm: universal property} is unique,
 the automorphism group of a compact complex supermanifold $\mathcal M$
 is the unique Lie supergroup 
 satisfying the universal property formulated in 
 Theorem~\ref{thm: universal property}.
\end{rmk}

\begin{rmk}
We say that a real Lie supergroup $\mathcal G$ acts on $\mathcal M$ by 
holomorphic transformations if the underlying Lie group $G$ acts on the complex
manifold $M$ by holomorphic transformations and if there is a homomorphism of
Lie superalgebras $\mathfrak g\rightarrow \mathrm{Vec}(\mathcal M)$ which is 
compatible with the action of $G$ on $M$.
Using the theory of Harish-Chandra pairs, we also have the Lie supergroup
$\mathcal G^\mathbb C$,
the universal complexification of $\mathcal G$; see \cite{Kalus}.
The underlying Lie group of $\mathcal G^\mathbb C$ is the universal 
complexification $G^\mathbb C$ of the Lie group $G$.
Let $\mathfrak g=\mathfrak g_{\bar 0}\oplus \mathfrak g_{\bar 1}$ denote the Lie
superalgebra of $\mathcal G$, $\mathfrak g_{\bar 0}$ the Lie algebra of 
$G$. Then the Lie algebra $\mathfrak g_{\bar 0}^\mathbb C$ of $G^\mathbb C$ is 
a quotient of $\mathfrak g_{\bar 0}\otimes\mathbb C$, and 
the Lie superalgebra of $\mathcal G^\mathbb C$ can be realised as 
$\mathfrak g_{\bar 0}^\mathbb C\oplus (\mathfrak g_{\bar 1}\otimes \mathbb C)$.
The action of $G$ on $\mathcal M$ extends to a holomorphic $G^\mathbb C$-action
on $\mathcal M$, and the homomorphism $\mathfrak g\rightarrow
\mathrm{Vec}(\mathcal M)$ extends to a homomorphism 
$\mathfrak g_{\bar 0}^\mathbb C\oplus (\mathfrak g_{\bar 1}\otimes \mathbb C)
\rightarrow \mathrm{Vec}(\mathcal M)$ of complex Lie superalgebras,
which is compatible with the $G^\mathbb C$-action on $\mathcal M$.
Thus, we have a holomorphic $\mathcal G^\mathbb C$-action on $\mathcal M$
extending the $\mathcal G$-action.
Moreover, there is a morphism $\sigma:\mathcal G^\mathbb C\rightarrow
\Aut(\mathcal M)$ of Lie supergroups as in Theorem~\ref{thm: universal property}.
\end{rmk}

\begin{ex}
 Let $\mathcal M=\mathbb C^{0|1}$.
 Denoting the odd coordinate on $\mathbb C^{0|1}$ by $\xi$,
 each super vector field on $\mathbb C^{0|1}$ is of the form
 $X=a \xi\frac{\partial}{\partial \xi}+ b\frac{\partial}{\partial\xi}$
 for $a,b\in\mathbb C$. The flow 
 $\varphi:\mathbb C\times \mathcal M\rightarrow\mathcal M$ of
 $a \xi\frac{\partial}{\partial \xi}$ is given by
 $(\varphi_t)^*( \xi)=e^{at}\xi$, 
 and the flow $\psi:\mathbb C^{0|1}\times\mathcal M\rightarrow \mathcal M$
 of $b\frac{\partial}{\partial\xi}$ by $\psi^*(\xi)=b\tau+\xi$.
 Let $X_0=\xi\frac{\partial}{\partial \xi}$ and
 $X_1=\frac{\partial}{\partial\xi}$.
 Then $\mathrm{Vec}(\mathbb C^{0|1})=\mathbb C X_0\oplus \mathbb C X_1=
 \mathbb C^{1|1}$, where the Lie algebra structure on $\mathbb C^{1|1}$
 is given by $[X_0,X_1]= -X_1$ and $[X_1,X_1]=0$.
 Note that this Lie superalgebra is isomorphic to the Lie superalgebra of 
 right-invariant vector fields on the Lie supergroup $(\mathbb C^{1|1},
 \mu_{0,1})$, where the multiplication $\mu=\mu_{0,1}$
 is given by $\mu^*(t)=t_1+t_2$ and $\mu^*(\tau)=\tau_1+e^{t_1} \tau_2$;
 for the Lie supergroup structures on $\mathbb C^{1|1}$ see e.g. 
 \cite{GarnierWurzbacher}, Lemma~3.1. In particular, the Lie superalgebra
 $\mathrm{Vec}(\mathbb C^{0|1})$ is not abelian.  
 
 Since each automorphism $\varphi$ of $\mathbb C^{0|1}$ is given by 
 $\varphi^*(\xi)=c\cdot\xi$ for some $c\in\mathbb C$, $c\neq 0$, 
 we have $\Aut_{\bar 0}(\mathbb C^{0|1})\cong \mathbb C^*$.
\end{ex}

\section{The case of a superdomain with bounded underlying domain}
 In the classical case, the automorphism group of a bounded domain 
 $U\subset \mathbb C^m$ is a (real) Lie group (see Theorem~13 in
 ``Sur les groupes de transformations analytiques'' in \cite{CartanOeuvres}).
 If $\mathcal U\subset \mathbb C^{m|n}$ is a superdomain whose underlying 
 set $U$ is a bounded domain in $\mathbb C^m$, it is in general not 
 possible to endow its set of automorphisms with the structure of a 
 Lie group such that the action on $\mathcal U$ is smooth, as will be 
 illustrated in an example.
 In particular, there is no Lie supergroup satisfying the universal
 property as the automorphism group of a compact complex
 supermanifold $\mathcal M$ does as formulated in 
 Theorem~\ref{thm: universal property}.
 
\begin{ex}
 Consider the superdomain $\mathcal U$ of dimension $(1|2)$ with bounded 
 underlying domain $U\subset \mathbb C$.
 Let $z,\theta_1,\theta_2$ denote coordinates for $\mathcal M$.
 For any holomorphic function $f$ on~$U$, define the even super vector field
 $X_f=f(z)\theta_1\theta_2\frac{\partial}{\partial z}$.
 The reduced vector field $\tilde{X}_f=0$ is completely integrable and 
 thus the flow of $X_f$ can be defined on $\mathbb C\times \mathcal U$
 (c.f. \cite{GarnierWurzbacher} Lemma~5.2).
 The flow is given by $(\varphi_t)^*(z)=z+t\cdot f(z)\theta_1\theta_2$ 
 and $(\varphi_t)^*(\theta_j)=\theta_j$.
 For any holomorphic functions $f$ and $g$ we have $[X_f,X_g]=0$, and thus
 their flows commute (c.f. \cite{Bergner}, Corollary~3.8).
 Therefore, $\{X_f|\,f\in\mathcal O(U)\}\cong \mathcal O(U)$ is an uncountable 
 infinite-dimensional abelian Lie algebra. If the set of automorphisms of 
 $\mathcal U$ carried the structure of a Lie group such that its action on 
 $\mathcal U$ was smooth, its Lie algebra would necessarily contain
 $\{X_f|\,f\in\mathcal O(U)\}\cong \mathcal O(U)$ as a Lie subalgebra, 
 which is not possible.
\end{ex}

\section{Examples}

In this section, we determine the automorphism group $\Aut(\mathcal M)$
for some complex supermanifolds $\mathcal M$ with underlying manifold
$M=\mathbb P_1\mathbb C$.

Let $L_1$ denote the hyperplane bundle on $M=\mathbb P_1\mathbb C$ with 
sheaf of sections $\mathcal O(1)$,
and $L_k=(L_1)^{\otimes k}$ the line bundle of degree $k$, $k\in\mathbb Z$,
on $\mathbb P_1\mathbb C$, and sheaf of sections $\mathcal O(k)$.
Each holomorphic vector bundle on $\mathbb P_1\mathbb C$ is isomorphic
to a direct sum of line bundles $L_{k_1}\oplus\ldots\oplus L_{k_n}$ (see
\cite{Grothendieck}).
Therefore, if $\mathcal M$ is a split supermanifold with 
$M=\mathbb P_1\mathbb C$ and $\dim \mathcal M=(1|n)$,
there exist $k_1,\ldots, k_n\in \mathbb Z$ such that
the structure sheaf $\mathcal O_{\mathcal M}$ of $\mathcal M$
is isomorphic to 
$$\bigwedge (\mathcal O(k_1)\oplus\ldots\oplus \mathcal O(k_n)).$$
Let $U_j=\{[z_0:z_1]\in \mathbb P_1\mathbb C\,|\,
z_j\neq 0\}$, $j=1,2$, and $\mathcal U_j=(U_j,\mathcal O_{\mathcal M}|_{U_j})$.
Moreover, define ${U_0}^*=U_0\setminus \{[1:0]\}$ and 
${U_1}^*=U_1\setminus\{[0:1]\}$, and let  
${\mathcal U_j}^*=({U_j}^*,\mathcal O_{\mathcal M}|_{{U_j}^*})$.
We can now choose local coordinates $z,\theta_1,\ldots, \theta_n$ for 
$\mathcal M$ on $U_0$, and local coordinates 
$w,\eta_1,\ldots,\eta_n$ on $U_1$
so that the transition map 
$\chi:{\mathcal U_0}^*\rightarrow {\mathcal U_1}^*$, which determines
the supermanifold structure of~$\mathcal M$, is given by
$$\chi^*(w)=\frac 1 z \ \ \text{ and }\ \ 
\chi^*(\eta_j)={z^{k_j}}\theta_j.$$

\begin{ex}
 Let $\mathcal M=(\mathbb P_1\mathbb C, \mathcal O_\mathcal M)$ be a complex 
 supermanifold with $\dim\mathcal M=(1|1)$. 
 Since the odd dimension is $1$, the supermanifold $\mathcal M$ has to be split.
 Let $-k\in \mathbb Z$ be the degree of the associated line bundle.
 Choose local coordinates $z,\theta$ for $\mathcal M$ on $U_0$ and $w,\eta$ 
 on $U_1$ as above so that the transition map
 $\chi:{\mathcal U_0}^*\rightarrow {\mathcal U_1}^*$ is given by
 $\chi^*(w)=\frac 1 z$ and $\chi^*(\eta)=\frac{1}{z^k}\theta$.
 
 We first want to determine the Lie superalgebra $\mathrm{Vec}(\mathcal M)$
 of super vector fields on $\mathcal M$.
 A calculation in local coordinates verifying the compatibility condition 
 with the transition map~$\chi$ yields that 
 the restriction to $U_0$ of any super vector field on $\mathcal M$ is of 
 the form 
 $$\left((\alpha_0+\alpha_1z+\alpha_2z^2)\frac{\partial}{\partial z}
 +(\beta+k\alpha_2 z)\theta\frac{\partial}{\partial \theta}\right)
 +\left(p(z)\frac{\partial}{\partial\theta}
 +q(z)\theta\frac{\partial}{\partial z}\right),$$
 where $\alpha_0,\alpha_1,\alpha_2,\beta\in\mathbb C$,
 $p$ is a polynomial of degree at most $k$, and $q$ is a polynomial
 of degree at most $2-k$. If $k<0$ (respectively $2-k<0$), the polynomial~$p$
 (respectively $q$) is~$0$.
 The Lie algebra $\mathrm{Vec}_{\bar 0}(\mathcal M)$
 of even super vector fields is isomorphic to 
 $\mathfrak{sl}_2(\mathbb C)\oplus \mathbb C$,
 where an isomorphism $\mathfrak{sl}_2(\mathbb C)\oplus \mathbb C \rightarrow
 \mathrm{Vec}_{\bar 0}(\mathcal M)$ is given by
 $$\left(\left(\begin{matrix}
          a&b\\c&-a
         \end{matrix}\right), d\right)
   \mapsto
   (-b-2az+cz^2)\frac{\partial}{\partial z}
   +((d-ka)+kcz)\theta\frac{\partial}{\partial \theta}.$$

Note that since the odd dimension of $\mathcal M$ is $1$ each automorphism 
$\varphi:\mathcal M\rightarrow\mathcal M$ gives rise to an automorphism 
of the line bundle $L_{-k}$ and vice versa. 
Hence, the automorphism group $\Aut(L_{-k})$ of the line bundle $L_{-k}$ and 
$\Aut_{\bar 0}(\mathcal M)$ coincide.

A calculation yields that the group $\Aut_{\bar 0}(\mathcal M)$
of automorphisms $\mathcal M\rightarrow\mathcal M$ can be identified 
with $\mathrm{PSL}_2(\mathbb C)\times \mathbb C^*$ if $k$ is even
and with $\mathrm{SL}_2(\mathbb C)\times \mathbb C^*$ if $k$ is odd.
Consider the element
$\left(\left(\begin{smallmatrix} a & b\\ c& d\end{smallmatrix}\right),s\right)$,
where $s\in\mathbb C^*$ and 
$\left(\begin{smallmatrix} a & b\\ c& d\end{smallmatrix}\right)$ is
either an element of $\mathrm{SL}_2(\mathbb C)$ or the representative of the 
corresponding class in $\mathrm{PSL}_2(\mathbb C)$.
The action of the corresponding element $\varphi\in\Aut_{\bar 0}(\mathcal M)$ 
on $\mathcal M$ is then given by
$$\varphi^*(z)=\frac{c+dz}{a+bz}\ \ \text{ and } \ \ 
\varphi^*(\theta)=\left(\frac{1}{(a+bz)^k}+s\right)\theta$$
as a morphism over appropriate subsets of $U_0$ and by
$$\varphi^*(w)=\frac{aw+b}{cw+d}\ \ \text{ and } \ \
\varphi^*(\eta)=\left(\frac{1}{(cw+d)^k}+s\right)\eta$$
over appropriate subsets of $U_1$.

The Lie supergroup structure on $\Aut(\mathcal M)$ is now uniquely determined
by $\Aut_{\bar 0}(\mathcal M)$, $\mathrm{Vec}(\mathcal M)$, and the adjoint
action of
$\Aut_{\bar 0}(\mathcal M)$ on $\mathrm{Vec}(\mathcal M)$.
Since $\Aut_{\bar 0}(\mathcal M)$ is connected, it is enough to 
calculate the adjoint action of $\mathrm{Vec}_{\bar 0}(\mathcal M)
\cong \mathfrak{sl}_2{\mathbb C} \oplus \mathbb C$ on 
$\mathrm{Vec}_{\bar 1}(\mathcal M)$.

Let $P_l$ denote the space of polynomials of degree at most $l$, and 
set $P_l=\{0\}$ for $l<0$.
The space of odd super vector fields $\mathrm{Vec}_{\bar 1}(\mathcal M)$
is isomorphic to $P_{k}\oplus P_{2-k}$ via 
$\left(p(z)\frac{\partial}{\partial\theta}
 +q(z)\theta\frac{\partial}{\partial z}\right)
 \mapsto (p(z),q(z))$.

The element $H=\left(\begin{smallmatrix}
                1&0\\0&-1
               \end{smallmatrix}\right)
\in \mathfrak{sl}_2(\mathbb C)\subset \mathfrak{sl}_2(\mathbb C)\oplus\mathbb C
\cong\mathrm{Vec}_{\bar 0}(\mathcal M)$
corresponds to 
$-2z\frac{\partial}{\partial z}-k\theta\frac{\partial}{\partial \theta}$.
The adjoint action of this super vector field on the first factor $P_k$ of 
$\mathrm{Vec}_{\bar 1}(\mathcal M)$ is given by
by $-2z\frac{\partial}{\partial z}+k\cdot\mathrm{Id}$, and
on the second factor $P_{2-k}$ by $-2z\frac{\partial}{\partial z}+ (2-k)\cdot
\mathrm{Id}$.
Calculating the weights of the $\mathfrak{sl}_2(\mathbb C)$-representation on
$P_k$ and $P_{2-k}$, we get that $P_k$ is the unique irreducible 
$(k+1)$-dimensional representation and $P_{2-k}$ the unique irreducible
$(3-k)$-dimensional representation.
Moreover, a calculation yields that $d\in\mathbb C$ corresponding to 
$d\cdot\theta\frac{\partial}{\partial\theta}\in\mathrm{Vec}_{\bar 0}(\mathcal M)$ 
acts on
$P_k$ by multiplication with~$-d$ and on $P_{2-k}$ by multiplication with~$d$.

If $k<0$ or $k>2$, we have 
$$[\mathrm{Vec}_{\bar 1}(\mathcal M), \mathrm{Vec}_{\bar 1}(\mathcal M)]=0.$$
In the case $k=0$, we have $P_k\cong \mathbb C$. Since 
$[\frac{\partial}{\partial \theta},q(z)\theta\frac{\partial}{\partial z}]
=q(z)\frac{\partial}{\partial z}$ for any $q\in P_2$, we
get 
$$[\mathrm{Vec}_{\bar 1}(\mathcal M), \mathrm{Vec}_{\bar 1}(\mathcal M)]
=\left.\left\{a(z)\frac{\partial}{\partial z}\,\right|\,a\in P_2\right\}
\cong \mathfrak{sl}_2(\mathbb C),$$
and the map $P_0\times P_2\rightarrow \mathrm{Vec}_{\bar 0}(\mathcal M)$,
$(X,Y)\mapsto [X,Y]$,
corresponds to $\mathbb C\times P_2\rightarrow \mathrm{Vec}_{\bar 0}(\mathcal M)$,
$(p,q(z))\mapsto p\cdot q(z)\frac{\partial}{\partial z}$.

Similarly, if $k=2$, we have $P_{2-k}\cong \mathbb C$, and 
$$[\mathrm{Vec}_{\bar 1}(\mathcal M), \mathrm{Vec}_{\bar 1}(\mathcal M)]
=\left\{\left. (\alpha_0+\alpha_1z+\alpha_2 z^2)\frac{\partial}{\partial z}
+(\alpha_1+2\alpha_2 z)\theta\frac{\partial}{\partial \theta}\,\right|\,
\alpha_0,\alpha_1,\alpha_2\in\mathbb C\right\}
\cong\mathfrak{sl}_2(\mathbb C)$$ since
$[p(z)\frac{\partial}{\partial \theta}, \theta\frac{\partial}{\partial z}]
=p(z)\frac{\partial}{\partial z}+
p'(z)\theta\frac{\partial}{\partial\theta}$,
and the map $P_2\times P_0\rightarrow\mathrm{Vec}_{\bar 0}(\mathcal M)$,
$(X,Y)\mapsto [X,Y]$,
corresponds to $P_2\times\mathbb C\rightarrow \mathrm{Vec}_{\bar 0}(\mathcal M)$,
$(p(z),q)\mapsto q\cdot p(z)\frac{\partial}{\partial z}+q\cdot p'(z)\theta
\frac{\partial}{\partial \theta}$.

If $k=1$, then $P_k\oplus P_{2-k}\cong \mathbb C^2\oplus \mathbb C^2$.
We have
$$\begin{array}{lcrclcr}
 \left[\frac{\partial}{\partial \theta},\theta\frac{\partial}{\partial z}\right]
&=&\frac{\partial}{\partial z},& 
&\left[z\frac{\partial}{\partial \theta},\theta\frac{\partial}{\partial z}\right]
&=&z\frac{\partial}{\partial z}+\theta\frac{\partial}{\partial \theta},\\
&&&&&&\\
\left[\frac{\partial}{\partial \theta},z\theta\frac{\partial}{\partial z}\right]
&=&z\frac{\partial}{\partial z},&
&\left[z\frac{\partial}{\partial \theta},z\theta\frac{\partial}{\partial z}\right]
&=&z^2\frac{\partial}{\partial z}+z\theta\frac{\partial}{\partial \theta},
\end{array}$$
and consequently 
$[\mathrm{Vec}_{\bar 1}(\mathcal M),\mathrm{Vec}_{\bar 1}(\mathcal M)]
=\mathrm{Vec}_{\bar 0}(\mathcal M)$.

Remark that $\Aut(\mathcal M)$ carries the structure of a split Lie supergroup
if and only if $k<0$ or $k>2$ (c.f. Proposition~$4$ in \cite{Vishnyakova}).
\end{ex}

\begin{ex}
 Let $\mathcal M=(\mathbb P_1\mathbb C,\mathcal O_{\mathcal M})$
 be a split complex supermanifold of dimension $\dim\mathcal M=(1|2)$
 associated to  $\mathcal O(-k_1)\oplus \mathcal O(-k_2)$, 
 $k_1,k_2\in\mathbb Z$.
 We will determine the group $\Aut_{\bar 0}(\mathcal M)$ of
 automorphisms $\mathcal M\rightarrow\mathcal M$.
 
 We choose coordinates $z,\theta_1,\theta_2$ for $\mathcal U_0$ and 
 $w,\eta_1,\eta_2$ for $\mathcal U_1$ as described above such that 
 the transition map $\chi$ is given by 
 $\chi^*(w)=z^{-1}$ and $\chi^*(\eta_j)={z^{-k_j}}\theta_j$.
 
 The action of $\mathrm{PSL}_2(\mathbb C)$ on $\mathbb P_1\mathbb C$ 
 by M\"obius transformations lifts to an action of $\mathrm{SL}_2(\mathbb C)$
 on $\mathcal M$ by letting $A=\left(\begin{smallmatrix}
                                    a&b\\ c&d
                                   \end{smallmatrix}\right)
 \in \mathrm{SL}_2(\mathbb C)$
 act by the automorphism $\varphi_A:\mathcal M\rightarrow \mathcal M$
 with pullback
 $$\varphi_A^*(z)=\frac{c+dz}{a+bz}\ \ \text{ and }\ \ 
 \varphi_A^*(\theta_j)=(a+bz)^{-k_j}\theta_j$$
 as a morphism over appropriate subsets of $U_0$, and
 $$\varphi_A^*(w)=\frac{aw+b}{cw+d}\ \ \text{ and }\ \
 \varphi_A^*(\eta_j)=(cw+d)^{-k_j}\eta_j$$
 over appropriate subsets of $U_1$. Using the transition map $\chi$ one might
 also calculate the representation of $\varphi$ in coordinates as a morphism 
 over subsets $U_0\rightarrow U_1$ and $U_1\rightarrow U_0$.
 
 If $k_1$ and $k_2$ are both even, we have $\varphi_A=\mathrm{Id}_\mathcal M$ 
 for $A=\left(\begin{smallmatrix}
                                    -1&0\\ 0&-1
                                   \end{smallmatrix}\right)$
 and thus we get an action of $\mathrm{PSL}_2(\mathbb C)$
 on $\mathcal M$.
 
 Consider the homomorphism of Lie groups $\Psi:\Aut_{\bar 0}(\mathcal M)
 \rightarrow\Aut(\mathbb P_1\mathbb C)$ assigning to each automorphism
 $\varphi:\mathcal M\rightarrow \mathcal M$ the underlying biholomorphic
 map $\tilde\varphi:\mathbb P_1\mathbb C\rightarrow \mathbb P_1\mathbb C$.
 This homomorphism $\Psi$ is surjective since $\Aut(\mathbb P_1\mathbb C)
 \cong\mathrm{PSL}_2(\mathbb C)$ and since the 
 $\mathrm{PSL}_2(\mathbb C)$-action on $\mathbb P_1\mathbb C$ lifts to an action 
 (of $\mathrm{SL}_2(\mathbb C)$) on the supermanifold $\mathcal M$.
 The kernel $\ker \Psi$ of the homomorphism $\Psi$ consists of those
 automorphisms $\varphi:\mathcal M\rightarrow\mathcal M$ whose underlying map
 $\tilde\varphi$ is the identity 
 $\mathbb P_1\mathbb C\rightarrow\mathbb P_1\mathbb C$.
 This kernel $\ker \Psi$ is a normal subgroup, $\mathrm{SL}_2(\mathbb C)$
 acts on $\ker\Psi$, and we have
 $$\Aut_{\bar 0}(\mathcal M)\cong\ker\Psi\rtimes \mathrm{SL}_2(\mathbb C)$$
 if $k_1$ and $k_2$ are not both even, and
 $\Aut_{\bar 0}(\mathcal M)\cong\ker\Psi\rtimes \mathrm{PSL}_2(\mathbb C)$
 if $k_1$ and $k_2$ are even.
 Thus, it remains to determine $\ker\Psi$.
 
 Let $\varphi:\mathcal M\rightarrow\mathcal M$ be an automorphism with 
 $\tilde\varphi=\mathrm{Id}$.
 Let $f$ and $b_{jk}$, $j,k=1,2$, be holomorphic functions on 
 $U_0\cong \mathbb C$
 such that the pullback of $\varphi$ over $U_0$ is given by
 $$\varphi^*(z)=z+f(z)\theta_1\theta_2\ \ \text{ and }\ \ 
 \varphi^*(\theta)=B(z)\theta,$$
 where $B(z)=\left(\begin{smallmatrix}
                    b_{11}(z)&b_{12}(z)\\
                    b_{21}(z)&b_{22}(z)
                   \end{smallmatrix}\right)$
 and $\varphi^*(\theta)=B(z)\theta$ is an abbreviation for 
 $\varphi^*(\theta_j)=b_{j1}(z)\theta_1+b_{j2}(z)\theta_2$ for $j=1,2$.
 Similarly, let $g$ and $c_{jk}$ be holomorphic functions on $U_1\cong\mathbb C$
 such that the pullback of $\varphi$ over $U_1$ is given by
 $$\varphi^*(w)=w+g(w)\eta_1\eta_2\ \ \text{ and }\ \
 \varphi^*(\eta)=C(z)\eta,$$
 where $C(z)=\left(\begin{smallmatrix}
                    c_{11}(z)&c_{12}(z)\\
                    c_{21}(z)&c_{22}(z)
                   \end{smallmatrix}\right)$.
 The compatibility condition with the transition map $\chi$
 gives now the relation 
 $$f(z)=-z^{2-(k_1+k_2)}g\left(\frac 1 z\right)\ 
 \text{ for  } z\in\mathbb C^*.$$
 Therefore, $f$ and $g$ are both polynomials of degree at most $2-(k_1+k_2)$,
 and they are $0$ in the case $k_1+k_2>2$.
 For the matrices $B$ and $C$ we get the relation
 $$B(z)=\begin{pmatrix}
               z^{k_1}&0\\0&z^{k_2}
              \end{pmatrix}C\left(\frac 1 z\right)
              \begin{pmatrix}
               z^{-k_1}&0\\0&z^{-k_2}
              \end{pmatrix}\ \text{ for } z\in\mathbb C^*.
$$
If $k_1=k_2$, this implies $B(z)=C\left(\frac 1 z\right)$ for all 
$z\in\mathbb C^*$. 
Thus, $B(z)=B$ and $C(w)=C$ are constant matrices, and $B=C\in 
\mathrm{GL}_2(\mathbb C)$ since $\varphi$ was assumed to be invertible.
Consequently, 
we have $$\ker \Psi\cong P_{2-(k_1+k_2)}\rtimes \mathrm{GL}_2(\mathbb C)$$
in the case $k_1=k_2$, 
where $P_{2-(k_1+k_2)}$ denotes the space of polynomials of degree at 
most $2-(k_1+k_2)$ if $k_1+k_2< 2$ and $P_{2-(k_1+k_2)}=\{0\}$ otherwise.
The group structure on the semidirect product is given by
$(f_1(z),B_1)\cdot (f_2(z),B_2)=(\det B_1 f_1(z)+f_2(z), B_1B_2)$.

Let now $k_1\neq k_2$. After possibly changing coordinates we may assume
$k_1>k_2$.
Then we have 
$$B(z)
 =\begin{pmatrix}
               z^{k_1}&0\\0&z^{k_2}
              \end{pmatrix}C\left(\frac 1 z\right)
              \begin{pmatrix}
               z^{-k_1}&0\\0&z^{-k_2}
              \end{pmatrix}
 =\begin{pmatrix}
   c_{11}\left(\frac 1 z\right) & z^{k_1-k_2}c_{12}\left(\frac 1 z\right)\\
   z^{k_2-k_1}c_{21}\left(\frac 1 z\right) & c_{22}\left(\frac 1 z\right)
  \end{pmatrix}$$
for all $z\in \mathbb C^*$.
This implies that $b_{11}=c_{11}$ and $b_{22}=c_{22}$ are constants.
Since we assume $k_1>k_2$, we also get $b_{21}=c_{21}=0$ and 
$b_{12}$ and $c_{12}$ are polynomials of degree at most $k_1-k_2$.
Therefore,
$$\ker \Psi\cong P_{2-(k_1+k_2)}\rtimes
\left\{ \left.\begin{pmatrix}
         \lambda & p(z)\\
         0&\mu
        \end{pmatrix}
 \,\right|\, \lambda,\mu\in\mathbb C^*,\, p\in P_{k_1-k_2}\right\},$$
 and the group structure is again given by 
 $(f_1(z),B_1)\cdot (f_2(z),B_2)=(\det B_1 f_1(z)+f_2(z), B_1B_2)$
 for $f_1,f_2\in P_{2-(k_1+k_2)}$, 
 $B_1,B_2\in \left\{ \left.\left(\begin{smallmatrix}
         \lambda & p(z)\\
         0&\mu
        \end{smallmatrix}\right)
 \,\right|\, \lambda,\mu\in\mathbb C^*,\, p\in P_{k_1-k_2}\right\}$.
 
 The semidirect product $\ker\Psi\rtimes \mathrm{SL}_2(\mathbb C)$
 (or $\ker\Psi\rtimes \mathrm{PSL}_2(\mathbb C)$)
 is a direct product if and only if
 $k_1=k_2$ and $k_1+k_2\geq 2$.
\end{ex}

\begin{ex}
 Let $\mathcal M=(\mathbb P_1\mathbb C, \mathcal O_\mathcal M)$
 be the complex supermanifold of dimension $\dim\mathcal M=(1|2)$ 
 given by the transition map $\chi:{\mathcal U_0}^*\rightarrow {\mathcal U_1}^*$
 with pullback 
 $$\chi^*(w)=\frac 1 z +\frac{1}{z^3}\theta_1\theta_2\ \ \text{ and }\ \ 
 \chi^*(\eta_j)=\frac{1}{z^2} \theta_j.$$
 The supermanifold $\mathcal M$ is not split
 and the associated split supermanifold corresponds to  
 $\mathcal O(-2)\oplus \mathcal O(-2)$; see e.g. \cite{BuneginaOnishchik}.
 
 As in the previous example, the action of $\mathrm{PSL}_2(\mathbb C)$ on 
 $\mathbb P_1\mathbb C$ by M\"obius transformations lifts
 to an action of $\mathrm{PSL}_2(\mathbb C)$ on $\mathcal M$.
 Let $A$ denote the class  of
 $\left(\begin{smallmatrix}
               a&b\\ c&d
              \end{smallmatrix}\right)\in\mathrm{SL}_2(\mathbb C)$
               in $\mathrm{PSL}_2(\mathbb C)$.
 Then $A$ acts by the morphism $\varphi_A:\mathcal M\rightarrow\mathcal M$
 whose pullback as a morphism over appropriate subsets of $U_0$ is given by
 $$\varphi_A^*(z)=\frac{c+dz}{a+bz}-\frac{b}{(a+bz)^3}\theta_1\theta_2
 \ \ \text{ and }\ \ 
 \varphi_A^*(\theta_j)=\frac{1}{(a+bz)^2}\theta_j.$$
 
 Let $\Psi:\Aut_{\bar 0}(\mathcal M)\rightarrow \Aut(\mathbb P_1\mathbb C)
 \cong \mathrm{PSL}_2(\mathbb C)$
 denote again the Lie group homomorphism which assigns to an automorphism of
 $\mathcal M$ the underlying automorphism of $\mathbb P_1\mathbb C$.
 The assignment $A\mapsto \varphi_A\in \Aut_{\bar 0}(\mathcal M)$
 defines a section $\mathrm{PSL}_2(\mathbb C)\rightarrow\Aut_{\bar 0}(\mathcal M)$
 of $\Psi$, and we have
 $$\Aut_{\bar 0}(\mathbb C)\cong \ker\Psi\rtimes \mathrm{PSL}_2(\mathbb C).$$
 The section $\mathrm{PSL}_2(\mathbb C)\rightarrow \Aut_{\bar 0}(\mathcal M)$
 induces on the level of Lie algebras the morphism
 $\sigma:\mathfrak{sl}_2(\mathbb C)\hookrightarrow \mathrm{Vec}_{\bar 0}(\mathcal M)$,
 which maps an element $\left(\begin{smallmatrix}
  a&b\\
  c&-a
 \end{smallmatrix}\right)\in\mathfrak{sl}_2(\mathbb C)$ to 
 the super vector field on $\mathcal M$ whose restriction to $\mathcal U_0$ is
 $$ \left(c-2az-bz^2-b\theta_1\theta_2\right)\frac{\partial}{\partial z}
 -2(a+bz)\left(\theta_1\frac{\partial}{\partial \theta_1}
+\theta_2\frac{\partial}{\partial \theta_2}\right).$$

 We now calculate the kernel $\ker\Psi$. Let $\varphi\in\ker\Psi$.
 Its underlying map $\tilde\varphi$ is the identity and we 
 thus have 
 $$\varphi^*(z)=z+a_0(z)\theta_1\theta_2\ \ \text{ and }\ \ 
 \varphi^*(\theta)=A_0(z)\theta$$
 on $U_0$ and 
 $$\varphi^*(w)=w+a_1(w)\eta_1\eta_2\ \ \text{ and }\ \ 
 \varphi^*(\eta)=A_1(w)\eta$$
 on $U_1$ for holomorphic functions $a_0$ and $a_1$ and 
 invertible matrices $A_0$ and $A_1$ whose entries are holomorphic functions.
 The notation $\varphi^*(\theta)=A_0(z)\theta$ (and similarly 
 $\varphi^*(\eta)=A_1(w)\eta$) 
 is again an abbreviation for 
 $\varphi^*(\theta_j)=(A_0(z))_{j1}\theta_1+(A_0(z))_{j2}\theta_2$,
 where $A_0(z)=\left((A_0(z))_{jk}\right)_{1\leq j,k\leq 2}$.
 A calculation with the transition map $\chi$ then yields
 the relations
 $$A_1(w)=A_0\left(\frac 1 w\right)
 \ \ \text{ and }\ \ 
 a_1(w)=\frac 1 w \left(\left(\det A_0\left(\frac 1 w \right)-1\right)
 -\frac 1 w a_0\left(\frac 1 w \right)\right)$$
 for any $w\in\mathbb C^*$.
 Since $a_0$, $a_1$, $A_0$, and $A_1$ are holomorphic on $\mathbb C$,
 we get that $A_0=A_1$ are constant matrices,
 $\det A_0=1$, and $a_0=a_1=0$.
 Therefore, 
 $\ker\Psi\cong \mathrm{SL}_2(\mathbb C)$,
 and its Lie algebra is 
 $$\mathrm{Lie}(\ker\Psi)
 =\left\{ \left.\left(a_{11}\theta_1+a_{12}\theta_2\right)
 \frac{\partial}{\partial \theta_1} +\left(a_{21}\theta_1+a_{22}\theta_2\right)
 \frac{\partial}{\partial \theta_2}\,\right|\, \begin{pmatrix} a_{11}&a_{12}\\
 a_{21}&a_{22} \end{pmatrix}\in\mathfrak{sl}_2(\mathbb C)\right\}.$$
 Since $\mathrm{Lie}(\ker\Psi)$ and 
 $\sigma(\mathrm{Lie}(\mathrm{PSL}_2(\mathbb C))$
 commute, the semidirect product $\ker\Psi\rtimes \mathrm{PSL}_2(\mathbb C)$
 is direct and we have
 $$\Aut_{\bar 0}(\mathcal M)\cong \mathrm{SL}_2(\mathbb C)\times 
 \mathrm{PSL}_2(\mathbb C).$$
 Remark in particular that this group is different from 
 the automorphism group of the corresponding split supermanifold $\mathcal N$, 
 which is associated to 
 $\mathcal O(-2)\oplus\mathcal O(-2)$, 
 with $\Aut_{\bar 0}(\mathcal N)
 \cong\mathrm{GL}_2(\mathbb C)\times \mathrm{PSL}_2(\mathbb C)$.
 \end{ex}

\bibliography{LiteraturAutomorphismengruppe.bib}

\end{document}